\documentclass[8pt]{article}
\usepackage[utf8]{inputenc}
\usepackage[usenames,dvipsnames,svgnames,table]{xcolor}	
\usepackage[]{amsmath}
\usepackage{cases}
\usepackage{empheq}
\usepackage[colorlinks=true]{hyperref}
\hypersetup{colorlinks,
citecolor=brown, linkcolor=blue, linktocpage
}
\usepackage{amssymb}
\usepackage{bbm}
\usepackage{enumerate}   
\usepackage{geometry}
\usepackage{amsfonts}
\usepackage{color}
\usepackage[]{amsthm} 
 \usepackage[shortlabels]{enumitem}
 \usepackage{cleveref}

\newtheorem{theorem}{Theorem}[section]
\newtheorem{lemma}[theorem]{Lemma}
\newtheorem{cor}[theorem]{Corollary}
\newtheorem{proposition}[theorem]{Proposition}
\newtheorem{mydef}[theorem]{Definition}

\newtheorem{remark}[theorem]{Remark}
\newtheorem*{hyp}{\textit{Hypothesis}}
\def\E{\mathbb{E}}
\def\R{\mathbb{R}}
\def\P{\mathbb{P}}
\def\Q{\mathbb{Q}}
\def\N{\mathbb{N}}

\def\1{\mathbbm{1}}
\def\F{\mathcal{F}}

\def\D1{\frac{\partial}{\partial x}}

\newcommand{\Cc}{\mathcal C}

\newcommand{\Pp}{\mathcal P}

\usepackage{xspace}

\title{A new McKean-Vlasov stochastic interpretation of the 
parabolic-parabolic Keller-Segel model: The two-dimensional case. }
\author{Milica Toma\v{s}evi\'c\footnote{ CMAP, CNRS, Ecole Polytechnique, Institut Polytechnique de Paris, 
91128 Palaiseau, France; milica.tomasevic@polytechnique.edu.} }
\date{}

\begin{document}
\maketitle
\noindent
\textbf{Abstract:} In Talay and Tomasevic \cite{Mi-De} we proposed a new stochastic interpretation of the parabolic-parabolic Keller-Segel system without cut-off via a McKean-Vlasov stochastic process. 
The latter was defined through an original type of 
interaction kernel which involved, in a singular way, all its past time marginals.
In the present paper, we study this McKean-Vlasov representation in the two-dimensional case. In this setting, there exists a possibility of a blow-up in finite time for the Keller-Segel system if some parameters of the model are large. Indeed, we prove the global in time well-posedness of the McKean-Vlasov process under some constraints involving a parameter of the model and the initial datum. Under these constraints, we also prove the global well-posedness for the Keller-Segel model in the plane. \\
\textbf{Key words:}  Keller–Segel system; Singular McKean-Vlasov non-linear stochastic differential equation.\\
\textbf{Classification:} 60H30, 60H10.


\section{Introduction}
The standard $d$-dimensional parabolic--parabolic Keller--Segel model for chemotaxis reads
\begin{equation}
\label{KSd}
\begin{cases}
&  \partial_t \rho(t,x)=\nabla \cdot (\frac{1}{2}\nabla \rho -  
\chi\rho \nabla 
c)(t,x), \quad t>0, \ x \in \R^d, \\
 & \alpha~\partial_t c(t,x) = \frac{1}{2}\triangle c(t,x)  -\lambda 
 c(t,x) + \rho(t,x),  \quad 
 t>0, \ x \in \mathbb{R}^d.\\
 & \rho(0,x)=\rho_0(x),~~~c(0,x)=c_0(x).
\end{cases}
\end{equation}  
It describes the time evolution of the density $\rho_t$ of 
a cell population and of the concentration $c_t$ of a chemical 
attractant. The parameter $\chi>0$ is called the chemotactic sensitivity and, together with the total mass $M:=\int \rho_0(x)\ dx$, plays an important role in the well-posedness theory for \eqref{KSd}. For a very thorough review of this theory for the standard
Keller--Segel model and its variations, see the reviews of Horstmann \cite{Horstmann1,Horstmann2}. For a review of the results obtained in past 15 years, see e. g. Toma\v{s}evi\'c \cite{Mi-De-2D}.

When $\alpha = 0$, the system is in its parabolic-elliptic version, while when $\alpha=1$ (more generally $\alpha> 0$) the system is known as the parabolic-parabolic Keller-Segel model.

Our goal is to provide, in an original manner, a probabilistic interpretation of the parabolic-parabolic Keller-Segel system and new existence and uniqueness results for the PDE in $d=2$. We emphasize here that this is the first result of this type on the standard parabolic-parabolic Keller-Segel model. That is, to the best of our knowledge, our stochastic representation of the system~\eqref{KSd} has not been studied so far in the literature. 

In the Mc-Kean Vlasov context, the originality of our stochastic interpretation comes from its unusual interaction with the law of the process. Namely, the process interacts with all the past time marginals of its probability distribution through a functional involving a singular kernel. That is why it cannot be analysed by means of standard coupling methods or Wasserstein distance contractions. 


As we will see in the next section, with our choice of space for the family $(\rho_t)_{t\geq 0}$,  the drift of the process will not be in the framework of Krylov and R\"ockner \cite{KryRoc-05}. That is, it will not be possible to control the drift with a function whose $L^p((0,T); L^q(\R^2))$-norm is finite for $1/p+1/q<1/2$. Nevertheless, we will construct a solution to the associated non-linear martingale problem (NLMP). Thus, our strategy will not be based on the strategy in~\cite{KryRoc-05}.

In Talay and Tomasevic \cite{Mi-De}, we propose the following stochastic representation of the parabolic-parabolic version of \eqref{KSd}: The family $(\rho_t)_{t\geq 0}$ is seen as a family of one dimensional time marginal distributions of the law of a McKean-Vlasov stochastic process and the family $(c_t)_{t \geq 0}$ as its transformation. Namely, we consider the following stochastic differential equation:
\begin{equation}
\label{NLSDEd}
\begin{cases}
& dX_t= b^{(d)}_0(t,X_t)dt + \Big\{ \int_0^t (K^{(d)}_{t-s}\ast 
\rho_s)(X_t)ds\Big\} dt
+  dW_t, \quad t> 0, \\
& \rho_s(y)dy:=  \mathcal{L}( X_{s}),\quad X_0 \sim \rho_0(x)dx,
\end{cases}
\end{equation}
where $$K^{(d)}_t(x):=\chi  e^{-\lambda t}\nabla( g^{(d)}_t(x)) ~~\text{and}~~ b^{(d)}_0(t,x):=\chi e^{-\lambda t}  
\nabla (c_0\ast g^{(d)}_t(x)).$$ Here, $g_t^{(d)}(x):=\frac{1}{(2\pi 
t)^{d/2}}e^{-\frac{|x|^2}{2t}}$, $(W_t)_{t\geq 0}$ is a $d$-dimensional Brownian motion on a 
filtered probability space $(\Omega, \F, \P, (\F_t) )$ and  $X_0$ is an 
$\R^d$-valued $\F_0-$measurable random variable. 
Then, we define
$$ c(t,x) :=e^{-\lambda t}(g^{(d)}_t \ast c_0)(x) + \int_0^t 
 \rho_{t-s}\ast e^{-\lambda s} g^{(d)}_s(x)~ds. $$
 The couple $(\rho,c)$ is our stochastic interpretation of the system in \eqref{KSd}.
 
%
%
%

The unusual interaction with the law of the process leads to a system of interacting particles with an uncommon interaction. Namely, the particle system is of of non-Markovian nature where each particle interacts in the current time with all the past of all the other particles through a singular functional involving the kernel $K^{(d)}$.
 

In \cite{Mi-De}, the authors overcome the above mentioned difficulties and validate the above stochastic interpretation in the framework of $d=1$ with no constrains on the parameters of the model. That is to say, under the assumptions that $\rho_0$ is a finite measure on $\R$ and  that $c_0$ belongs to $\Cc_b^1(\R)$, the global (in time) well-posedness of the Mc-Kean Vlasov SDE \eqref{NLSDEd} in $d=1$ is proven. As a consequence, the same holds for the Keller-Segel PDE. This result generalizes the results in Osaki and Yagi \cite{OsakiYagi} and Hillen and Potapov \cite{HillenPotapov}. 
One of the key points in \cite{Mi-De} is that the one--dimensional kernel $(K^{(1)})$ belongs to the space $L^1((0,T); L^1(\R))$.  Then, by the help of precise $L^\infty(\R)$-norm density estimates the singularity is tamed. Namely, Picard's iteration procedure is used to exhibit a weak solution to \eqref{NLSDEd}. In each step the $L^\infty([0,T]\times \R)$-norm of the drift and $L^1((0,T];L^\infty(\R))$-norm of the marginal densities were controlled simultaneously. These controls were obtained thanks to a probabilistic method which exhibits sharp density estimates for one dimensional It\^o processes with bounded drift term (see Qien and Zheng \cite{QianZheng} and \cite{Mi-De}). 

Furthermore, in $d=1$,
Jabir \textit{et. al} \cite{JTT} prove the well-posedness of the interacting particle system associated to \eqref{NLSDEd} in $d=1$ and its propagation of chaos. The non-Markovian nature of the system is treated with techniques based on Girsanov theorem (see \cite{KryRoc-05}). The calculation was
based on the fact that the kernel $K^{(1)}$ is in $L^1((0,T);L^2(\R))$. In order to get the tightness in number of particles and the propagation of chaos, the authors needed to use the so called ``partial" Girsanov transformations removing a finite number of particles from the system.
 
 Contrary to the one-dimensional case, a blow up may occur for the Keller-Segel system in the two-dimensional setting if the parameter $\chi$ is large. In the parabolic-elliptic version of the system, its behaviour has been completely understood. That is, the system exhibits the ``threshold" behaviour: if $M\chi<8\pi$ the solutions are global in time, if $M\chi>8\pi$ every solution blows-up in finite time  (see e.g. Blanchet \textit{et. al} \cite{Blanchet} and Nagai and Ogawa~\cite{NagaiOgawa}).
 
The fully parabolic model does not exhibit the same bahaviour. It has been proved that when $M\chi<8\pi$ one has global existence (see Calvez and Corrias \cite{CalCor2008} and Mizogouchi \cite{Mizo}). However, in Biller \textit{et. al} \cite{BilerCorrias} the authors find an initial configuration of the system in which a global solution in some sense exists  with $M\chi>8\pi$. Then, Herrero and Vel\'azquez \cite{HerreroandVelazquez} construct a radially symmetric solution on a disk that blows-up and develops $\delta$-function type singularities when $M\chi>8\pi$. Finally, unique solution with any positive mass exists when the the parameter $\alpha$ is large enough (Corrias \textit{et. al} \cite{Corrias2014}). Thus, in the case of parabolic-parabolic model, the value $8\pi$ can somehow be understood as a threshold, but in a different sense: under it there is global existence, over it there exists a solution that blows up.

 

In order to shed a new light on the parabolic-parabolic model
 in $d=2$, we study our stochastic representation.
To obtain an existence result for it,
 we will perform a regularization of the singular interaction kernel and combine probabilistic and PDE techniques to exhibit a solution to the NLMP related to \eqref{NLSDEd}.
A condition on the size of parameter $\chi$ will be necessary to obtain some drift and density estimates for the regularized process that are uniform w.r.t. regularization parameter.

 A consequence of the existence result for \eqref{NLSDEd} is the global existence for~\eqref{KSd} in $d=2$. This generalizes the result in \cite{Corrias2014} by removing the assumption on the smallness of the initial datum (for more details see the next section).
 
 The uniqueness of the constructed solution to \eqref{KSd} in $d=2$ is obtained under an additional condition on the size of parameter $\chi$. Then, the uniqueness for NLMP is obtained from uniqueness of the solution to the linearised MP. The latter comes from the so-called transfer of uniqueness from the (linear) Fokker-Planck equation to the (linear) martingale problem (see Trevisan~\cite{Trevisan} and the references therein). 
 
 We conclude this part with a remark concerning the interacting particle system associated to \eqref{NLSDEd} in $d=2$. As on the mean-field level we will work with a drift function that does not fit in the framework of~\cite{KryRoc-05}, the technique, based on Girsanov transformation, used in \cite{JTT} 
  will not work in $2$-d case. The existence of solutions to the non-regularized particle system and its propagation of chaos are still work in progress.
%
%
 
 The plan of the paper is the following: In Section~\ref{sec:main_results} we state our main results and compare them with the above mentioned Keller-Segel literature. In Section~\ref{chd2:Regularization} we present our regularization procedure and we obtain density estimates of the regularized processes independent of the regularization parameter. These estimates enable us to prove in Section~\ref{chd2:Thmproof} 
  the existence of a solution to the NLMP corresponding to \eqref{NLSDEd}.
   Then, in the same section 
 we prove the global existence for the Keller-Segel system in $d=2$. 
 Section \ref{sec:uniq} is devoted to uniqueness. Finally, in Appendix~\ref{sec:1chsmooth} we prove the well-posedness of a smoothed version of~\eqref{NLSDEd}.
 
 In all the paper $C$ will denote a generic constant, $C_p$ will denote a constant depending on a parameter~$p$ and $C(p_1, p_2, \dots)$ will denote a constant depending on parameters $p_1, p_2, \dots$ In addition, from now on we will drop the dimension index in the definitions of the interaction kernel ($K^{(d)}$) and linear drift ($b_0^{(d)}$). 
 \section{Main results}
 \label{sec:main_results}
Let $T>0$. On a filtered probability space $(\Omega, \F, \P, (\F_t) )$, equipped with a $2$-dimensional Brownian motion $(W_t)_{0\leq t \leq T}$, we consider  the following non-linear stochastic process:
\begin{equation}
\label{NLSDEd2}
\begin{cases}
 &dX_t=dW_t + b_0(t,X_t)dt +\chi \Big\{\int_0^t  e^{-\lambda(t-s)} (K_{t-s}\ast p_s)(X_t) \ ds\Big\} dt, \\
& p_s(y)dy:=\mathcal{L}( X_{s}),\quad X_0 \sim \rho_0.
\end{cases}
\end{equation}
Here $X_0$ is an 
$\R^2$-valued $\F_0-$measurable random variable, $g_t$ denotes the probability density of $W_t$ and for $(t,x)\in (0,T]\times \R^2$ we denote
\begin{equation}
\label{def:b0andK}
b_0(t,x):=\chi e^{-\lambda t}(\nabla c_0\ast g_t )(x) \text{ and } K_t(x):= \nabla g_t(x)= -\frac{x}{2 \pi t^2}e^{-\frac{|x|^2}{2t}}.
\end{equation}
Note that $K_t$ is a two dimensional vector. We denote its coordinates by $K^i_t$ with $i=1,2$ and 
$$b^i(t,x;p) :=b^i_0(t,X_t) +\chi \int_0^t e^{-\lambda (t -s)}\int K^i_{t-s}(X_t-y)p_s(y)dy \ ds\ dt. $$
In order to prove the (weak) well-posedness of \eqref{NLSDEd2}, we will solve the associated NLMP. By classical arguments, one can then pass from a solution to this martingale problem to the existence of a weak solution to \eqref{NLSDEd2} (see e.g.~\cite{KaratzasShreve}).

An important issue when formulating the NLMP is to choose an appropriate functional space for the family $(p_s)_{s\geq 0}$ that would enable one to claim that $b(t,x;p)$ is well-defined and integrable in time.
 
 
In one-dimensional framework, $L^1((0,T);L^\infty(\R))$ turned out to be an appropriate functional space for the marginals. However, the increase in the space dimension leads to an increase in the strength of the singularity of the interaction kernel $K$. This has a significant impact on the techniques used in~\cite{Mi-De} to prove the well-posedness of \eqref{NLSDEd} in $d=1$.

  To see this, 
 a generalisation to the multidimensional case of the results in~\cite{QianZheng} can be found in Qian \textit{et al.} \cite{QianRussoZheng} in the case of time homogeneous drifts. There, the authors show that the estimate of the transition density of a  $d$-dimensional stochastic process is a product of one-dimensional estimates provided that the Euclidean norm of the drift vector is uniformly bounded. With the arguments we used in $d=1$, one can extend the results in \cite{QianRussoZheng} to time inhomogeneous drifts. However, under reasonable conditions on $\rho_0$ and $\chi$,  one can only construct local solutions to \eqref{NLSDEd2} using the strategy in~\cite{Mi-De}.  
 
 That is why, a new functional space for the marginals needs to be introduced and the next obvious choice are the $L^q$-spaces. Analysing a priori the mild equation for the one-dimensional time marginals $p_t$, one can note that if $\rho_0 \in L^1(\R^2)$, the term depending on the initial condition gives an $L^q$-norm in space of order $\sim t^{-(1-1/q)}$. Thus, if we were to impose the following Gaussian behaviour for $p_t$: $\sup_{t\leq T} t^{1-1/q}\|p_t\|_{L^q(\R^2)}<\infty $ for a $q>2$, then one would obtain that $\|b(t,\cdot;p)\|_{L^\infty(\R^2)}$ is of order $\sim t^{-\frac{1}{2}}$. 
 
 Therefore, we define the NLMP related to \eqref{NLSDEd2} as follows:
 \begin{mydef}
\label{defMP2d}
Consider the canonical space $\mathcal{C}([0,T];\R^2)$ equipped with its canonical filtration. Let $\Q$ be a probability measure on this canonical space and denote by $Q_t$ its one dimensional time marginals. $\Q$ solves the non-linear 
martingale problem $(MP)$ if:
\begin{enumerate}[(i)]
\item $\Q_0 = \rho_0$.
\item For any $t\in (0,T]$, $\Q_t$ have densities $q_t$ w.r.t. Lebesgue measure on $\R^2$. In addition, they satisfy 
\begin{equation}
\label{dens_Est_d2}
\forall 1<r< \infty ~ \exists  C_r(\chi)>0,~~~\sup_{t\leq T}t^{1-\frac{1}{r}}\|q_t\|_{L^r(\R^2)}\leq C_r(\chi).
\end{equation}
\item For any $f \in C_K^2(\R^2)$ the process $(M_t)_{t\leq T}$, 
defined as
\begin{multline*}
M_t:=f(w_t)-f(w_0)-\int_0^t \Big[\frac{1}{2} \triangle f(w_u)+\nabla f(w_u)\cdot \Big(b_0(u,w_u) \\ +\chi \int_0^u \int K_{u- \tau}(w_u-y) q_\tau(y)dyd\tau\Big)\Big]du
\end{multline*}
is a $\Q$-martingale where $(w_t)$ is the canonical process.  
\end{enumerate}
\end{mydef} 
\begin{remark}Under the condition $c_0 \in H^1(\R^2)$ one has, applying H\"{o}lder's inequality and \eqref{gauss2dLq},
$$\int_0^t|b_0(s,x)|~ds \leq C \|\nabla c_0\|_{L^2(\R^2)} \int_0^t\frac{1}{\sqrt{s}}~ds.$$
Similarly, \eqref{dens_Est_d2} and \eqref{kerneld2Lp} imply 
$$\int_0^t\left|\int_0^s K_{s-u} \ast q_u (x) \ du\right|~ds \leq  C \int_0^t\frac{1}{\sqrt{s}}~ds.$$
Moreover, 
if $c_0 \in H^1(\R^2)$ and if  \eqref{dens_Est_d2} holds, then 
\begin{equation}
\label{driftSpacePreliminary}
\forall 2\leq r\leq \infty ~ \exists  C_r>0,~~~\sup_{t\leq T}t^{\frac{1}{2}-\frac{1}{r}} \|b(t,\cdot;q)\|_{L^r(\R^2)}\leq C_r.
\end{equation}
\end{remark}
Now, we present our first main result.
\begin{theorem}
\label{ch:nlsdeD2th1}
Let $\lambda \geq 0$, $T>0$ and fix a $q\in (2,4)$.  Suppose that $\rho_0$ is a density function. 
 Furthermore, assume that $c_0 \in H^1(\R^2)$. Then, \hyperref[defMP2d]{$(MP)$} admits a solution under the following condition
 \begin{equation}
\label{d2exCond}
A \chi \|\nabla c_0\|_{L^2(\R^2)} + B\sqrt{\chi}<1,
\end{equation}
where 
\begin{align*}
& A =  C_1(q')  C_2(\frac{2q}{q+2})  \beta\left(\frac{3}{2}-\frac{2}{q}, \frac{3}{2}-\frac{1}{q'}\right),\\
& B= 2\sqrt{C_2(q) C_1(q') C_1(1) \beta\left(\frac{3}{2}-\frac{2}{q}, \frac{3}{2}-\frac{1}{q'}\right) \beta \left(1-\frac{1}{q}, \frac{1}{2}\right)}.
\end{align*}
Here  $q'$ is such that $\frac{1}{q}+\frac{1}{q'}=1$ and the functions $C_1(\cdot), C_2(\cdot)$ and $\beta(\cdot,\cdot)$ are defined in Lemma~\ref{lemma:kerneld2Lp}, Lemma~\ref{lemma:gauss2dLq} and Eq.~\eqref{def:beta}, respectively.
\end{theorem}
 To prove the above result, we do not apply Picard's iteration procedure since in each iteration step we will need a well-posedness result for a linear SDE whose drift satisfies \eqref{driftSpacePreliminary}. In view of Krylov and R\"{o}ckner \cite{KryRoc-05}, the well-posedness follows from a finite $L^q((0,T);L^r(\R^2))$-norm of the drift with $1/q + 1/r<1/2$.
 Unfortunately, the property in \eqref{driftSpacePreliminary} will imply the opposite condition 
$1/p + 1/r>1/2$ for the same norm to be finite. 
 
 Thus, to prove Theorem~\ref{ch:nlsdeD2th1} we will
 use a regularization method. We prove that the time marginals of the regularized version of \eqref{NLSDEd2} satisfy the property \eqref{dens_Est_d2} with uniform constants with respect to the regularization parameter. That is where the condition \eqref{d2exCond} emerges. Then, the tightness will follow thanks to \eqref{driftSpacePreliminary} for $r=\infty$. It will remain, then, to show that \hyperref[defMP2d]{$(MP)$} admits a solution. The strong well-posedness of the regularized equation is an adaptation of the results in Sznitman~\cite{Sznitman} presented in a general way in the Appendix.
 
In addition, the incompatibility of~\eqref{driftSpacePreliminary} and the condition in \cite{KryRoc-05} makes  us doubt that Girsanov transform techniques would work and that the law of \eqref{NLSDEd2} is absolutely continuous w.r.t. Wiener's measure even under \eqref{d2exCond}.

The next objective is to use Theorem \ref{ch:nlsdeD2th1} to get the well-posedness of the Keller-Segel model in $d=2$. The system reads
\begin{subequations}
\label{KSd2}
\begin{empheq}[left={}\empheqlbrace]{align}
  & \frac{\partial\rho}{\partial t}(t,x) =\nabla \cdot (\frac{1}{2}\nabla \rho(t,x) - \chi \rho(t,x) \nabla c(t,x)), \quad  t>0,~~x \in \R^2, 
  \label{KSd21} 
  \\
  &  \frac{\partial c}{\partial t}(t,x) =  
  \frac{1}{2}\triangle c(t,x) - \lambda c(t,x) + \rho(t,x), \quad t>0,~~x \in \R^2, 
  \label{KSd22}\\
  &  \rho(0,x)= \rho_0(x),~~~c(0,x)= c_0, \nonumber
\end{empheq}
\end{subequations}
where $\chi>0$ and $\lambda\geq 0$. Notice that the two diffusion coefficient are deliberately chosen to be equal to $\frac{1}{2}$ in order to have unit diffusion coefficient and standard Gaussian kernel in the formulation of \eqref{NLSDEd2}.

The new functions $\tilde{\rho}(t,x) := \frac{\rho(t,x)}{M}$ and
$\tilde{c}(t,x) := \frac{c(t,x)}{M}$
satisfy the system~\eqref{KSd2} with the new 
parameter~$\tilde{\chi} := \chi M$. Therefore, w.l.o.g. we may and do thereafter assume that $M=1$. We consider the following notion of solution to \eqref{KSd2}:
\begin{mydef}
\label{notionOfSold2}
Given the functions $\rho_0$ and $c_0$, and the constants $\chi >0$, 
$\lambda\geq0$, $T>0$, the pair $(\rho,c)$ is said 
to be a 
solution to~\eqref{KSd2} if~$\rho(t,\cdot)$ is a probability density 
function for 
every $0\leq t\leq T$,
one has 
\begin{equation}
\label{rho_space}
\forall 1< q<\infty ~ \exists C_q(\chi) >0:\quad \sup_{t\leq T} t^{1-\frac{1}{q}}\|\rho(t,\cdot)\|_{L^{q}(\R^2)}\leq C_q(\chi),
\end{equation}
and the following equality
\begin{equation} 
\label{eq:rho-KSd2}
\rho(t,x) =g_t\ast \rho_0(x) - \chi \sum_{i=1}^2 \int_0^t\nabla_i g_{t-s} \ast (\nabla_i c(s,\cdot)
~\rho(s,\cdot))(x)~ds
\end{equation}
is satisfied in the sense of the distribution with
\begin{equation} \label{eq:c-KSd2}
c(t,x) = e^{-\lambda t}(g(t,\cdot \ )\ast c_0)(x)+ \int_0^t 
e^{-\lambda s} (g_s \ast \rho(t-s,\cdot))(x)~ds.
\end{equation}
\end{mydef}
Notice that the function $c(t,x)$ defined by~\eqref{eq:c-KSd2}
is a mild solution to~\eqref{KSd22}.
These solutions are known as integral solutions and they have 
already been studied in PDE literature for the two-dimensional 
Keller-Segel model (see~\cite{Corrias2014} and references therein). 

A consequence of Theorem \ref{ch:nlsdeD2th1} is the following result for \eqref{KSd2}:
\begin{theorem}
\label{existenceKSd2}
Let $T>0$, $\lambda \geq 0$ and $\chi >0$. Let $\rho_0$ a probability density function and $c_0 \in H^1(\R^2)$. Under the condition \eqref{d2exCond} the system~\eqref{KSd2} admits a global solution in the sense of Definition \ref{notionOfSold2}.
\end{theorem}
Let us compare the above result with the literature mentioned in the introduction. 
In  \cite{CalCor2008} the authors obtain the global existence in sub-critical case assuming:
\begin{enumerate}[i)]
\item $\rho_0 \in L^1(\R^2)\cap L^1(\R^2, \log(1+|x|^2)dx) $ and $\rho_0 \log \rho_0 \in L^1(\R^2)$;
\item $c_0\in H^1(\R^2)$ if $\lambda >0$ or $c_0 \in L^1(\R^2)$ and $|\nabla c_0|\in L^2(\R^2)$ if $\lambda=0$;
\item $\rho_0 ~ c_0 \in L^1(\R^2)$.
\end{enumerate}
We emphasize that their sub-critical condition translates into $4 \chi < 8\pi $ for \eqref{KSd2} due to the additional diffusion coefficients in it and the assumption $M=1$. In the same sub-critical case, the global existence result is obtained in \cite{Mizo} assuming $\rho_0 \in L^1(\R^2)\cap L^\infty(\R^2)$ and $c_0\in H^1(\R^2) \cap  L^1(\R^2) $. Our result does not assume any additional conditions other than that $\rho_0$ is a density function and  $c_0\in H^1(\R^2)$. The price to pay is the condition \eqref{d2exCond} that not just involves the parameter $\chi$, but the initial datum as well. 

It is more appropriate to compare Theorem \ref{existenceKSd2} to the result in \cite[Thm. 2.1]{Corrias2014}. Indeed, the assumptions on initial conditions are the same and as well the notion of solution. However, the setting and the objectives are different. In \cite{Corrias2014}, the parameter $\alpha$ (see \eqref{KSd}) is not fixed to be equal to $1$ and plays an important role. The goal is to prove the global existence for any positive mass $M$ and $\chi=1$. This is achieved under the following conditions: 
\begin{enumerate}[C1:]
\item There exists $\delta= \delta (M, \alpha)$ such that $\|\nabla c_0\|_{L^2(\R)}< \delta$,
\item There exists a constant $C=C(\alpha)$ such that $M<C(\alpha)$.
\end{enumerate}
The condition C2 is similar to the Condition \eqref{d2exCond} for $\chi$. We cannot totally compare them as the constants are not explicitely written in \cite{Corrias2014}. What is important is that $C(\alpha)$ grows with $\alpha$, so one can have $M$ as large as one likes  in C2 as soon as $\alpha$ is large enough as well. In the present paper our objective is to get results for the classical Keller-Segel model ($\alpha =1$) with respect to the chemo-attractant sensitivity. When we assume the same ($\alpha =1$, $M=1$ and $\chi>0$) in the framework of \cite{Corrias2014}, we see that we have removed the assumption on the smallness of the initial datum (C1). The reason lies in our method: in \cite{Corrias2014} Banach's fixed point is used to construct a solution locally in time (where C1 emerges) and, then, such solution is globalized (where C2 emerges). In our case only a condition of C2 type appears as, thanks to our regularization procedure, we directly construct a global solution. In addition, as our condition is explicitly written in Section \ref{chd2:Regularization}, one can analyse the constants in order to find the optimal condition on $\chi$.

Our next main result concerns the uniqueness of the constructed solutions.
\begin{theorem}
\label{th:uniqKS}
Let the assumptions of Theorem \ref{ch:nlsdeD2th1} hold.
Then, the Keller-Segel system \eqref{KSd2} admits a unique solution in the sense of Definition~\ref{notionOfSold2} provided that $\chi$ satisfies the following additional condition:
\begin{equation}
\label{chiCondUniq}
C_0\chi( \|\nabla c_0\|_{L^2(\R^2)}+ B_q(\chi)) < 1.
\end{equation}
Here, $C_0$ is a universal constant and $B_q(\chi)$ is given in \eqref{dens_estFIXq}. 
\end{theorem}

Finally, using the so-called transfer of uniqueness we prove the uniqueness of the solution to \hyperref[defMP2d]{$(MP)$}. Namely, we will use the results in Trevisan \cite{Trevisan} to prove the following theorem:
\begin{theorem}
\label{UniquenessLawD2}
Let the assumptions of Theorem \ref{th:uniqKS} hold. The martingale problem \hyperref[defMP2d]{$(MP)$} admits a unique solution.
\end{theorem}
To summarize, the organization of our proofs is as follows: Firstly, we show the existence of a solution to the problem \hyperref[defMP2d]{$(MP)$}; Secondly, we prove existence of a solution to \eqref{KSd2}; Thirdly, the uniqueness is proved for \eqref{KSd2}; Finally, we get the uniqueness of solution to \hyperref[defMP2d]{$(MP)$}.
\section{Regularization procedure and density estimates}
 \label{chd2:Regularization}
 In this section we present our regularization procedure. Then, we get estimates on the marginals of the regularized process that are uniform in the regularization parameter. That is where we derive the explicit condition \eqref{d2exCond} on the size of $\chi$. We start with a preliminary section.
 \subsection{Preliminaries}
 The following lemmas will be used throughout the paper. 
\begin{lemma}
\label{lemma:kerneld2Lp}
Let $t>0$ and $i\in\{1,2\}$. Then, for any $1\leq q <\infty$ one has
\begin{equation}
\label{kerneld2Lp}
\|K_{t}^i\|_{L^q(\R^2)}=\|\nabla_i g_t\|_{L^q(\R^2)}= \frac{C_1(q)}{t^{\frac{3}{2}-\frac{1}{q}}},
\end{equation}
where 
$$C_1(q)= \frac{2^{\frac{1}{q}-\frac{1}{2}}}{\pi^{1-\frac{1}{2q}} q^{\frac{1}{q}+\frac{1}{2}}}\left(\Gamma(\frac{q+1}{2})\right)^{\frac{1}{q}}.$$
Here $\Gamma(x)$ denotes the Gamma function: $\Gamma(x)=\int_0^\infty z^{x-1}e^{-z} \ dz$.
\end{lemma}
Notice that for $q\geq 2$, the $L^1((0,T); L^q(\R^2))$--norm of $K^i$ explodes. On the other side, in $d=1$ the kernel belongs to $L^1((0,T); L^q(\R))$ for any $1\leq q <\infty$. It is in this sense that the two--dimensional kernel is more singular than the one-dimensional one.
\begin{proof}
Let $1\leq q <\infty$. One has
\begin{align*}
&\|K_{t}^i\|_{L^q(\R^2)}=\|\nabla_i g_t\|_{L^q(\R^2)}= \frac{1}{2 \pi t^2}\left( \int_{\R^2} |x_i|^q e^{-\frac{q|x|^2}{2t}}dx\right)^{\frac{1}{q}}\\
&=\frac{1}{2 \pi t^2}\left(\int_{\R} e^{-q\frac{x^2}{2t}}dx \int_{\R} |x|^q e^{-q\frac{x^2}{2t}}dx \right)^{\frac{1}{q}}=\frac{1}{2 \pi t^2} \left(  \frac{\sqrt{2\pi t}}{\sqrt{q}} \ 2 \int_0^\infty x^q e^{-q\frac{x^2}{2t}} \ dx\right)^{\frac{1}{q}}.
\end{align*}
Apply the change of variables $\frac{qx^2}{2t}= y$. It comes
\begin{align*}
& \|K_{t}^i\|_{L^q(\R^2)}=\|\nabla_i g_t\|_{L^q(\R^2)}= \frac{1}{2 \pi t^2} \left(  \frac{\sqrt{2\pi t}}{\sqrt{q}} \ 2  \left(\frac{2t}{q}\right)^{\frac{q-1}{2}}\int_0^\infty y^{\frac{q-1}{2}} e^{-y}  \frac{t}{q}\ dy\right)^{\frac{1}{q}}\\
&= \frac{1}{2 \pi t^2} \left( \frac{2t}{q}\right)^{\frac{1}{q}+\frac{1}{2}} \pi^{\frac{1}{2q}} \left(\Gamma(\frac{q+1}{2})\right)^{\frac{1}{q}}.
\end{align*}
This ends the proof.
\end{proof}

The change of variables $\frac{x}{\sqrt{t}}=z$ leads to the following:
\begin{lemma}
\label{lemma:gauss2dLq}
Let $t>0$. Then, for any $1\leq q <\infty$ one has
\begin{equation}
\label{gauss2dLq}
\|g_t\|_{L^q(\R^2)}=\frac{1}{(2\pi)^{1-\frac{1}{q}} q^{\frac{1}{q}}t^{1-\frac{1}{q}}}=:\frac{C_2(q)}{t^{1-\frac{1}{q}}}.
\end{equation}
\end{lemma}
The functions $C_1(q)$ and $C_2(q)$ will be used only when we need the explicit constants in a computation. In all other cases we will use the notation $C_q$ that may change from line to line. 

Now, for $0<a,b<1$, we denote
\begin{equation}
\label{def:beta}
\beta(a,b):= \int_0^1 \frac{1}{u^a \ (1-u)^b}du.
\end{equation}
The change of variables $\frac{s}{t}=u$ implies the following result:
\begin{lemma}
\label{lemma:stdComputation}
Let $t>0$ and $0<a,b<1$. Then,
$$\int_0^t \frac{1}{s^a(t-s)^b} \ ds = t^{1-(a+b)} \beta(a,b).$$
\end{lemma}
Next, we state here the two standard convolution inequalities in their general form. The following is proven in Brezis \cite[Thm. 4.15]{Brezis}:
\begin{lemma}[The convolution inequality]
Let $f\in L^p(\R^d)$ and $g\in L^1(\R^d)$ with $1\leq p\leq \infty$. Then, $f \ast g  \in L^p(\R^d)$ and 
\begin{equation}
\label{convIneq}
\|f \ast g \|_{L^p(\R^d)} \leq \|f\|_{L^p(\R^d)} \|g\|_{L^1(\R^d)}.
\end{equation}
\end{lemma}
The following is an extension of the preceding inequality (see \cite[Thm. 4.33]{Brezis}):
\begin{lemma}[The convolution inequality]
Let $f\in L^p(\R^d)$ and $g\in L^q(\R^d)$ with $1\leq p,q \leq \infty$ and $\frac{1}{r}=\frac{1}{q}+\frac{1}{p}-1\geq 0$. Then, $f \ast g  \in L^r(\R^d)$ and 
\begin{equation}
\label{convIneq2}
\|f \ast g \|_{L^r(\R^d)} \leq \|f\|_{L^p(\R^d)} \|g\|_{L^q(\R^d)}.
\end{equation}
\end{lemma}
Now, we are ready to prove the following lemma about the behaviour of the linear part of the drift:
\begin{lemma}
\label{lemma:linDrift}
Let $t>0$. Then, the function $b_0^i(t,\cdot)$ is continuous on $\R^2$ and for $r\in [2, \infty]$, one has
$$\|b_0^i(t,\cdot)\|_{L^r(\R^2)}\leq \chi  \|\nabla c_0\|_{L^2(\R^2)(\R^2)} \frac{C_2(\frac{2r}{r+2})}{t^{\frac{1}{2}-\frac{1}{r}}}.$$
\end{lemma}
\begin{proof}
As $\nabla_i c_0$ is only in $L^2(\R^2)$ we cannot apply the classical results of convolution with a continuous function. The continuity of $b_0^i(t,\cdot)= \chi \nabla_i c_0 \ast g_t$ is a direct consequence of \cite[Ex. 4.30-3.]{Brezis} as for a $t>0$ both $g_t$ and $\nabla_i c_0$ belong to $L^2(\R^2)$. 


Let $q\geq1$ be such that $\frac{1}{q}+\frac{1}{2}=1+\frac{1}{r}$. By the convolution inequality \eqref{convIneq2}, one has
$$\|b_0^i(t,\cdot)\|_{L^r(\R^2)}\leq \chi \|\nabla_i c_0\|_{L^2(\R^2)}\|g_t\|_{L^q(\R^2)}.$$
In view of estimates on $\|g_t\|_{L^q(\R^2)}$ and the relation above between $r$ and $q$, one has
$$\|b_0^i(t,\cdot)\|_{L^r(\R^2)}\leq \chi \|\nabla c_0\|_{L^2(\R^2)} \|g_t\|_{L^{\frac{2r}{r+2}}}\leq \chi \|\nabla c_0\|_{L^2(\R^2)} \frac{C_2(\frac{2r}{r+2})}{t^{1-(\frac{1}{r}+\frac{1}{2})}}.$$ 
\end{proof}
The following lemma is a direct application of Lemma 8 in Brezis and Cazenave~\cite{BrezisCazenave} with $N=2$ and $q=1$:
\begin{lemma}
\label{lemma:Brezis}
Let $p_0$ a probability density function on $\R^2$ and $1<r<\infty$.  One has
$$\limsup_{t\to 0} t^{1-\frac{1}{r}} \|g_t\ast p_0\|_{L^r(\R^2)}=0.$$
\end{lemma}
\subsection{Regularization}
 We define the regularized version of the interaction kernel $K$ and the linear part of the drift as follows. For $\varepsilon >0$ and $(t,x)\in (0,T)\times \R^2$ define
$$K_t^\varepsilon:= \frac{t^2}{(t+\varepsilon)^2}K_t(x), \ \ g_t^\varepsilon(x):= \frac{t}{ (t+\varepsilon)}g_t(x)   \ \ \text{ and }   \ \ b_0^\varepsilon(t,x):= \chi e^{-\lambda t}(\nabla c_0\ast g_t^\varepsilon )(x).$$
 For a $t\leq T$, the regularized Mc-Kean-Vlasov equation reads
 \begin{equation}
\label{NLSDEd2REG}
\begin{cases}
& dX_t^\varepsilon=dW_t + b_0^\varepsilon(t,X_t^\varepsilon)dt +\chi\{\int_0^t e^{-\lambda (t -s)}(K_{t-s}^\varepsilon\ast\mu_s^\varepsilon)(X_t^\varepsilon)\ ds\} dt,  \\
& \mu_s^\varepsilon:=  \mathcal{L}( X_{s}^\varepsilon),\quad X_0^\varepsilon \sim \rho_0,
\end{cases}
\end{equation}
Set
$$b^\varepsilon(t,x;\mu^\varepsilon):= b_0^\varepsilon(t,x) +\chi\int_0^t e^{-\lambda (t -s)}\int K_{t-s}^\varepsilon(x-y)\mu_s^\varepsilon(dy)\ ds.$$

Similar computations as the ones to get \eqref{kerneld2Lp} and \eqref{gauss2dLq} lead to the following estimates. For $t\in (0,T]$ and $1\leq q <\infty$, one has
\begin{equation}
\label{reg2dKernelEst}
\|K_{t}^{\varepsilon, i}\|_{L^q(\R^2)}\leq \frac{C_1(q)}{(t+\varepsilon)^{\frac{3}{2}-\frac{1}{q}}} \quad \text{ and } \quad \|g_t^\varepsilon\|_{L^q(\R^2)}\leq \frac{C_2(q)}{(t+\varepsilon)^{1-\frac{1}{q}}}.
\end{equation}
Repeating the arguments as in the proof of Lemma \ref{lemma:linDrift}, one gets
\begin{lemma}
\label{lemma:linDriftREG}
For $t>0$ and $r\in [2, \infty]$ one has
$$\|b_0^{\varepsilon, i}(t,\cdot)\|_{L^r(\R^2)}\leq \chi \|\nabla c_0\|_{L^2(\R^2)}  \frac{C_2(\frac{2r}{r+2})}{(t+\varepsilon)^{\frac{1}{2}-\frac{1}{r}}}.$$
\end{lemma}
\begin{proposition}
\label{sec3:prop1:regNLSDE}
Let $T>0, \chi >0, \lambda\geq 0$, $ c_0 \in H^1(\R^2)$ and $\rho_0$ be a density function on $\R^2$. Then, for any $\varepsilon>0$, Equation \eqref{NLSDEd2REG} admits a unique
 strong solution. Moreover, the one dimensional time marginals of the law of the solution admit probability density functions,  $(p_t^\varepsilon)_{t\leq T}$. In addition, for $t \in (0,T)$, $p_t^\varepsilon$ satisfies the following mild equation in the sense of the distributions:
\begin{equation}
\label{mildEqReg}
p_t^\varepsilon= g_t\ast \rho_0 - \sum_{i=1}^2 \int_0^t \nabla_i g_{t-s}\ast (p_s^\varepsilon b^{\varepsilon,i}(s,\cdot;p^\varepsilon))ds .
\end{equation}
\end{proposition}
\begin{proof}
It is clear that there exists $C_\varepsilon >0$ such that for any $ t \in (0,T)$ and any $x,y \in \R^2$, one has
\begin{equation*}
|b_0^\varepsilon(t,x)-b_0^\varepsilon(t,y)|+ |K_t^\varepsilon(x)-K_t^\varepsilon(y)|\leq C_\varepsilon|x-y|~~ \text{ and } ~~|b_0^\varepsilon(t,x)|+|K_t^\varepsilon(x)|\leq C_\varepsilon.
\end{equation*}
Thus, Theorem \ref{th:smoothInter} implies that the strong solution to Eq. \eqref{NLSDEd2REG} is uniquely well defined.
 In addition, as the drift term is bounded, we can apply Girsanov's transformation and conclude that the one dimensional time marginals of the law of the solution admit probability density functions. By classical arguments (see e.g. \cite{Mi-De}), one can prove that for $t \in (0,T)$, $p_t^\varepsilon$ satisfies \eqref{mildEqReg} in sense of the distributions.
\end{proof}
\subsection{Density estimates}
 For a $1<q<\infty$, let us define
\begin{equation}
\label{ch:d2:norm}
\mathcal{N}_q^\varepsilon(t):= \sup_{s\in (0,t)} s^{1-\frac{1}{q}}\|p_s^\varepsilon\|_{L^q(\R^2)}.
\end{equation}
The following lemma provides a first estimate for 
$\mathcal{N}_q^\varepsilon(t)$ for a fixed $\varepsilon>0$. This estimate is not the optimal one in $\varepsilon$, but it is necessary in order to be sure that all the quantities we work with are well defined. Also, it will be used in order to obtain the limit behaviour of $\mathcal{N}_q^\varepsilon(t)$ as $t\to 0$.
\begin{lemma}
\label{lemma:firstEst}
Let $0<t\leq T$ and $\varepsilon >0$ fixed. For any $1<q<\infty$, there exists $C_\varepsilon(T, \chi)>0$ such that 
\begin{equation}
\label{lemma:firstEst:1}
\mathcal{N}_q^\varepsilon(t)\leq C_\varepsilon(T, \chi).
\end{equation}
Moreover, one has
\begin{equation}
\lim_{t\to 0}\mathcal{N}_q^\varepsilon(t) =0.
\label{lemma:firstEst:2}
\end{equation}
\end{lemma}
As $K^\varepsilon$ is smooth, we can propose a simplified version of the arguments in \cite[p. 285-286]{BrezisCazenave} for the proof of \eqref{lemma:firstEst:2}.
\begin{proof}
The drift of the regularized stochastic equation is bounded. Indeed, $|K^\varepsilon_t| \leq \frac{C}{\varepsilon^{3/2}}$ and  Lemma \ref{lemma:linDriftREG} imply
$$\|b^{\varepsilon,i}(t,\cdot;p^\varepsilon)\|_{L^\infty(\R^2)}\leq \frac{C}{\sqrt{\varepsilon}} +  \frac{Ct}{\varepsilon^{3/2}} =: C_\varepsilon(1 + t).$$
For $1<q<\infty$ and $q'$ such that $\frac{1}{q}+\frac{1}{q'}=1$ integrate \eqref{mildEqReg} w.r.t. a test function $f\in L^{q'}(\R^2)$ and apply H\"{o}lder's inequality. It comes
\begin{multline}
\label{prop1:eq1}
\left| \int p_t^\varepsilon(x) f(x)dx \right|\leq \|f\|_{L^{q'}(\R^2)}\Big(\|g_t \ast \rho_0\|_{L^q(\R^2)}  \\
  + \sum_{i=1}^2\int_0^t \|\nabla_i g_{t-s}\ast (p_s^\varepsilon b_s^{\varepsilon,i})\|_{L^q(\R^2)}ds\Big).
\end{multline}
Now, we split the proof in two parts: $q\in(1,2)$ and $q\geq 2$.

 Assume $q\in(1,2)$.
The above drift bound and the convolution inequality \eqref{convIneq} applied in \eqref{prop1:eq1}, lead to
$$\|p_t^\varepsilon\|_{L^q(\R^2)}\leq \|g_t \ast \rho_0\|_{L^q(\R^2)} +  C_\varepsilon(1 + t)\sum_{i=1}^2\int_0^t \|\nabla_i g_{t-s}\|_{L^q(\R^2)}\|p_s^\varepsilon\|_{L^1(\R^2)} ds.$$
In view of \eqref{kerneld2Lp}, we deduce that
$$\int_0^t  \|\nabla_i g_{t-s}\|_{L^q(\R^2)}\|p_s^\varepsilon\|_{L^1(\R^2)} \ ds\leq C_q \int_0^t \frac{1}{(t-s)^{\frac{3}{2}-\frac{1}{q}}}= C_q t^{\frac{1}{q}-\frac{1}{2}}.$$
Thus,
\begin{equation}
\label{lemma:firstEst:eq1}
t^{1-\frac{1}{q}} \|p_t^\varepsilon\|_{L^q(\R^2)}\leq t^{1-1/q}\|g_t \ast \rho_0\|_{L^q(\R^2)} + t^{1-\frac{1}{q}+\frac{1}{q}-1/2} C_\varepsilon(1 + t).
\end{equation}
To get \eqref{lemma:firstEst:1}, in \eqref{lemma:firstEst:eq1} use the convolution inequality \eqref{convIneq} and that  $\|g_t\|_{L^q(\R^2)}= \frac{C}{t^{1-\frac{1}{q}}}$. To get \eqref{lemma:firstEst:2}  use Lemma \ref{lemma:Brezis} for the first term of the r.h.s. of \eqref{lemma:firstEst:eq1} and the fact that the second term tends to zero as $t\to 0$.

 Assume $q\geq 2$ and set $ \frac{1}{p_1}= \frac{1}{p_2}=\frac{1}{2}+ \frac{1}{2q}$. Then, $1<p_1, p_2<2$ and $1+\frac{1}{q}= \frac{1}{p_1}+\frac{1}{p_2}$. The convolution inequality \eqref{convIneq2} and the drift estimate above together with \eqref{prop1:eq1}, lead to
\begin{multline*}
\|p_t^\varepsilon\|_{L^q(\R^2)}\leq \|g_t \ast \rho_0\|_{L^q(\R^2)} \\+  C_\varepsilon(1 + T)\sum_{i=1}^2\int_0^t \|\nabla_i g_{t-s}\|_{L^{p_1}(\R^2)}\|p_s^\varepsilon\|_{L^{p_2}(\R^2)} ds.
\end{multline*}
In view of \eqref{kerneld2Lp} and the result for $q\in (1,2)$, one has
\begin{multline*}
t^{1-\frac{1}{q}}\|p_t^\varepsilon\|_{L^q(\R^2)}\leq t^{1-\frac{1}{q}}\|g_t \ast \rho_0\|_{L^q(\R^2)} +  \\C_\varepsilon(1 + T) t^{1-\frac{1}{q}} \int_0^t \frac{C_q}{(t-s)^{\frac{3}{2}-\frac{1}{p_1}}}\frac{C(\varepsilon, T)}{s^{1-\frac{1}{p_2}}}ds.
\end{multline*}
Apply Lemma \ref{lemma:stdComputation} and use the relation between the exponents. It comes:
$$t^{1-\frac{1}{q}}\|p_t^\varepsilon\|_{L^q(\R^2)}\leq t^{1-\frac{1}{q}}\|g_t \ast \rho_0\|_{L^q(\R^2)} +t^{1-\frac{1}{q}} \frac{C(\varepsilon, T)}{t^{\frac{1}{2}-\frac{1}{q}}} \beta(1-\frac{1}{p_2}, \frac{3}{2}-\frac{1}{p_1}).$$
To obtain the desired result, repeat the exact same steps as in the part $q\in (1,2)$ of the proof.
\end{proof}
 The following proposition enables one to control $\mathcal{N}_q^\varepsilon(t)$ for a fixed $q$ and uniformly in small $\varepsilon$.
\begin{proposition}
\label{prop1}
Let the assumptions of Theorem \ref{ch:nlsdeD2th1} hold.
Then, there exists $C>0$  such that for any $t\in (0,T]$, $\mathcal{N}_q^\varepsilon(t)$ defined in \eqref{ch:d2:norm} satisfies
$$ \forall 0<\varepsilon<1: ~~~ \mathcal{N}^\varepsilon_q(t)\leq C.$$
\end{proposition}
\begin{proof}
 Integrating \eqref{mildEqReg} w.r.t. a test function $f\in L^{q'}(\R^2)$, one starts from 
\begin{multline}
\label{prop2:eq1}
\left| \int p_t^\varepsilon(x) f(x)dx \right|\leq \|f\|_{L^{q'}(\R^2)}\\ \times \left(\|g_t \ast \rho_0\|_{L^q(\R^2)} + \sum_{i=1}^2\int_0^t \|\nabla_i g_{t-s}\ast (p_s^\varepsilon b_s^{\varepsilon,i})\|_{L^q(\R^2)}ds\right).
\end{multline}
Let us fix $i\in \{1,2\}$, $s<t$ and denote $A_s^i:=  \|\nabla_i g_{t-s}\ast (p_s^\varepsilon b_s^{\varepsilon,i})\|_{L^q(\R^2)}$. Observe that $\frac{1}{q'}+\frac{2}{q}=1+\frac{1}{q}$.  Apply the convolution inequality \eqref{convIneq2} and then use \eqref{kerneld2Lp}. It comes
 \begin{align*}
 A_s^i&\leq \|\nabla_i g_{t-s}\|_{L^{q'}(\R^2)}\|p_s^\varepsilon b_s^{\varepsilon,i}\|_{L^{\frac{q}{2}}(\R^2)}\\
 &\leq \frac{C_1(q') \|b_s^{\varepsilon,i}\|_{L^{q}(\R^2)}s^{1-\frac{1}{q}}\|p_s^\varepsilon\|_{L^{q}(\R^2)}}{(t-s)^{\frac{3}{2}-\frac{1}{q'}}s^{1-\frac{1}{q}}}\leq C_1(q') \mathcal{N}_q^\varepsilon(t) \frac{  \|b_s^{\varepsilon,i}\|_{L^{q}(\R^2)} }{(t-s)^{\frac{3}{2}-\frac{1}{q'}} s^{1-\frac{1}{q}}}.
 \end{align*}
In view of Lemma \ref{lemma:linDriftREG}, \eqref{reg2dKernelEst} and Lemma \ref{lemma:stdComputation}, we get
\begin{align*}
& \|b_s^{\varepsilon,i}\|_{L^{q}(\R^2)}\leq  \frac{C_2(\frac{2q}{q+2}) \chi \|\nabla c_0\|_{L^2(\R^2)} }{(s+\varepsilon)^{\frac{1}{2}-\frac{1}{q}}}+ \chi \int_0^s \|K^{i,\varepsilon}_{s-u}\|_{L^{1}(\R^2)}\|p_u^\varepsilon\|_{L^{q}(\R^2)}du\\
& \leq \frac{C_2(\frac{2q}{q+2}) \chi \|\nabla c_0\|_{L^2(\R^2)} }{s^{\frac{1}{2}-\frac{1}{q}}}+\chi C_1(1) \mathcal{N}_q^\varepsilon(t) \int_0^s  \frac{1}{\sqrt{s-u} \ u^{1-\frac{1}{q}}} \ ds\\
& \leq  \frac{C_2(\frac{2q}{q+2}) \chi \|\nabla c_0\|_{L^2(\R^2)} + \chi C_1(1) \mathcal{N}_q^\varepsilon(t) \beta (1-\frac{1}{q}, \frac{1}{2}) }{s^{\frac{1}{2}-\frac{1}{q}}}.
\end{align*}
It comes
$$A_s^i\leq C_1(q') \chi \mathcal{N}_q^\varepsilon(t) \frac{C_2(\frac{2q}{q+2})  \|\nabla c_0\|_{L^2(\R^2)} +  C_1(1) \mathcal{N}_q^\varepsilon(t) \beta (1-\frac{1}{q}, \frac{1}{2})  }{(t-s)^{\frac{3}{2}-\frac{1}{q'}} s^{\frac{3}{2}-\frac{2}{q}}}.$$
Plug this into \eqref{prop2:eq1}.  The condition $q\in(2,4)$
ensures that $\frac{3}{2}-\frac{2}{q}<1$ and $ \frac{3}{2}-\frac{1}{q'}<1$. Thus, Lemma~\ref{lemma:stdComputation} leads to
\begin{multline*}
\left| \int p_t^\varepsilon(x) f(x)dx \right|\leq \|f\|_{L^{q'}(\R^2)}\Big(\|g_t \ast \rho_0\|_{L^q(\R^2)} +2 C_1(q') \chi \mathcal{N}_q^\varepsilon(t)\\ \times \frac{C_2(\frac{2q}{q+2})  \|\nabla c_0\|_{L^2(\R^2)} +  C_1(1) \mathcal{N}_q^\varepsilon(t) \beta (1-\frac{1}{q}, \frac{1}{2})  }{t^{1-\frac{1}{q}}} \beta(\frac{3}{2}-\frac{2}{q}, \frac{3}{2}-\frac{1}{q'}) \Big).
\end{multline*}
Take $\sup_{\|f\|_{L^{q'}}=1}$ in the preceding inequality. 
It follows from the convolution inequality \eqref{convIneq} and \eqref{gauss2dLq} that
\begin{multline*}
 \|p_t^\varepsilon\|_{L^{q}(\R^2)}\leq \frac{C_2(q)}{t^{1-\frac{1}{q}}}+2C_1(q')\chi \beta(\frac{3}{2}-\frac{2}{q}, \frac{3}{2}-\frac{1}{q'}) \mathcal{N}_q^\varepsilon(t)\\
 \times \frac{C_2(\frac{2q}{q+2})  \|\nabla c_0\|_{L^2(\R^2)} +  C_1(1) \mathcal{N}_q^\varepsilon(t) \beta (1-\frac{1}{q}, \frac{1}{2})  }{t^{1-\frac{1}{q}}}. 
\end{multline*}
Let us denote 
\begin{align*}
&K_1:=2C_1(q') C_1(1) \beta(\frac{3}{2}-\frac{2}{q}, \frac{3}{2}-\frac{1}{q'}) \beta (1-\frac{1}{q}, \frac{1}{2}),\\
& K_2:=2 C_1(q')  C_2(\frac{2q}{q+2})  \beta(\frac{3}{2}-\frac{2}{q}, \frac{3}{2}-\frac{1}{q'}).
\end{align*}
After rearranging the terms, 
\begin{equation}
\label{d2:polynom}
0\leq K_1 \chi (\mathcal{N}_q^\varepsilon(t))^2+ (K_2 \chi \|\nabla c_0\|_{L^2(\R^2)}  -1) \mathcal{N}_q^\varepsilon(t) +C_2(q).
\end{equation}
Under the assumptions
$$K_2 \chi  \|\nabla c_0\|_{L^2(\R^2)}-1<0\text{ and }(K_2 \chi \|\nabla c_0\|_{L^2(\R^2)}-1)^2-4K_1C_2(q) \chi >0,$$
the polynomial function 
$$P(z)= K_1 \chi z^2 + (K_2 \chi \|\nabla c_0\|_{L^2(\R^2)}  -1)z +  C_2(q)$$
admits two positive roots. In view of Lemma \ref{lemma:firstEst} and the relation \eqref{d2:polynom}, one has that $\lim_{t\to 0}\mathcal{N}_q^\varepsilon(t)= 0$ and $P(\mathcal{N}_q^\varepsilon(t))>0$ for any $t\in [0,T]$. Necessarily, for any $t\in [0,T]$, $\mathcal{N}_q^\varepsilon(t)$ is bounded from above by the smaller root of the polynomial function $P(z)$. As the constants do not depend on $T$ and $\varepsilon$, this estimate is uniform in time and does not depend on the regularization parameter.

Note that the above condition is equivalent to 
$$K_2 \chi \|\nabla c_0\|_{L^2(\R^2)} + 2\sqrt{K_1C_2(q) \chi}<1.$$ 
Denote $A:=K_2$ and $B:=2\sqrt{C_2(q)K_1}$ to finish the proof.
\end{proof}
\begin{remark}
\label{explicitBoundd2}
Fix $q, A, B$ and $\chi$ as in Proposition \ref{prop1}. An upper bound $C$ of $\mathcal{N}_q^\varepsilon(t)$ is given by
\begin{equation}
\label{dens_estFIXq}
B_q(\chi):= \frac{1-A \chi \|\nabla c_0\|_{L^2(\R^2)}- \sqrt{(1-A \chi \|\nabla c_0\|_{L^2(\R^2)})^2-  B^2 \chi}}{2 K_1 \chi}.
\end{equation}
\end{remark}
Now, we will analyse $\mathcal{N}_r^\varepsilon(t)$, for different values of $r$ w.r.t. the $q\in (2,4)$ fixed in Proposition~\ref{prop1}. We will see that different arguments are used when $r<q$ and $r>q$ in order to control $\mathcal{N}_r^\varepsilon(t)$. The result obtained for $r<q$ will be used to control $\|b^{\varepsilon}_t\|_{L^r(\R^2)}$, for $r\geq 2$. All the estimates on $\mathcal{N}_r^\varepsilon(t)$ will be regrouped in the end of this section.
\begin{cor}
\label{cor:densSmallQ}
Let the assumptions of Theorem \ref{ch:nlsdeD2th1} hold. Then, for $1<r< q$, it holds
$$\forall ~ 0<\varepsilon < 1, ~~~\mathcal{N}_r^\varepsilon(T)\leq B_r(\chi).$$ 
\end{cor}
\begin{proof}
Let $1<r< q$. Define $\theta:= \frac{1-\frac{1}{r}}{1-\frac{1}{q}}$. Then, $\frac{1}{r}= 1-\theta + \frac{\theta}{q}$. As $p^\varepsilon_t \in L^1(\R^2)$, "interpolation inequalities" (see \cite[p. 93]{Brezis}) lead to
$$\|p^\varepsilon_t\|_{L^r(\R^2)}\leq \|p^\varepsilon_t\|_{L^1(\R^2)}^{1-\theta} \|p^\varepsilon_t\|_{L^q(\R^2)}^{\theta}\leq \frac{C^\theta}{t^{\theta(1- \frac{1}{q})}}=:\frac{C_r}{t^{1- \frac{1}{r}}}.$$
\end{proof}
\begin{cor}
\label{cor:driftEstREG}
Let the assumptions of Theorem \ref{ch:nlsdeD2th1} hold. Then, for $2\leq r \leq \infty$, 
\begin{equation*}
\forall ~ 0<\varepsilon < 1, ~~~\|b^{\varepsilon}_t\|_{L^r(\R^2)}\leq \frac{C_r(\chi, \|\nabla c_0\|_{L^2(\R^2)} )}{t^{\frac{1}{2}- \frac{1}{r}}}
\end{equation*}
\end{cor}
\begin{proof}
In view of Lemma \ref{lemma:linDriftREG}, one has for $i \in \{1,2\}$
\begin{equation}
\label{cor:driftEstREG:eq1}
\|b^{i,\varepsilon}_t\|_{L^r(\R^2)}\leq  \frac{C(\chi, \|\nabla c_0\|_{L^2(\R^2)} )}{t^{\frac{1}{2}- \frac{1}{r}}} + \chi \int_0^t\|K^{\varepsilon,i}_{t-s}\ast p_s^\varepsilon \|_{L^r(\R^2)}ds.
\end{equation}
Let $q\in (2,4)$ fixed in Proposition~\ref{prop1}. We split the proof in two parts: $r\in [2,q)$ and $r\in [q,\infty]$.

 For $r\in [2,q)$, Corollary \ref{cor:densSmallQ} immediately implies
$$\|K^{\varepsilon,i}_{t-s}\ast p_s^\varepsilon \|_{L^r(\R^2)}\leq \frac{C}{\sqrt{t-s} s^{1-\frac{1}{r}}}.$$

 For $ q\leq r \leq \infty$, choose $p_1$ such that $\frac{1}{p_1}:=1+\frac{1}{r}-\frac{1}{q}$. Observe that, as $2<q\leq r$, it follows that $\frac{1}{2}<\frac{1}{p_1}\leq 1$. Applying the convolution inequality \eqref{convIneq2} and Corollary \ref{cor:densSmallQ}, one has 
$$\|K^{\varepsilon,i}_{t-s}\ast p_s^\varepsilon \|_{L^r(\R^2)}\leq \frac{C}{(t-s)^{\frac{3}{2}-\frac{1}{p_1}} s^{1-\frac{1}{q}}}.$$

To finish the proof, in both cases, one plugs the obtained estimates in \eqref{cor:driftEstREG:eq1} and applies Lemma~\ref{lemma:stdComputation}.
\end{proof}
\begin{cor}
\label{cor:densBigQ}
Let the assumptions of Theorem \ref{ch:nlsdeD2th1} hold. Then, for $q<r < \infty$, one has
$$\forall ~ 0<\varepsilon < 1, ~~~ \mathcal{N}^\varepsilon_r(T)\leq B_r(\chi).$$
\end{cor}
\begin{proof}
Let $1<q_1,q_2<2$ such that $\frac{1}{q_1} = \frac{1}{q_2}= \frac{1}{2}+\frac{1}{2r}$. Then, $1+\frac{1}{r}= \frac{1}{q_1} + \frac{1}{q_2}$. Convolution inequality \eqref{convIneq2} leads to 
$$\|p_t^\varepsilon\|_{L^r(\R^2)}\leq \|g_t\ast \rho_0\|_{L^r(\R^2)} + \sum_{i=1}^2\int_0^t\|\nabla_i g_{t-s}\|_{L^{q_1}(\R^2)} \|p_s^\varepsilon b_s^{\varepsilon,i}\|_{L^{q_2}(\R^2)} \ ds.$$
Let us apply H\"{o}lder's inequality for $\frac{1}{\lambda_1}+\frac{1}{\lambda_2}=1$ such that $\lambda_1= \frac{q}{2}$,  
$$\|p_s^\varepsilon b_s^{\varepsilon,i}\|_{L^{q_2}(\R^2)} \leq \|p_s^\varepsilon\|_{L^{\lambda_1 q_2}(\R^2)} \| b_s^{\varepsilon,i}\|_{L^{\lambda_2 q_2}(\R^2)}.$$
Notice that $1<\lambda_1 < 2$ since $2<q<4$ by hypothesis. Then, $\lambda_2>2$, thus $\lambda_2 q_2 >2$. In addition, $\lambda_1 q_2= \frac{q}{2} q_2 <q$. In view of Corollaries \ref{cor:densSmallQ} and \ref{cor:driftEstREG}, one has
$$\|p_s^\varepsilon b_s^{\varepsilon,i}\|_{L^{q_2}(\R^2)} \leq \frac{C}{s^{1-\frac{1}{\lambda_1 q_2}+ \frac{1}{2} - \frac{1}{\lambda_2 q_2}}}=\frac{C}{s^{\frac{3}{2} - \frac{1}{q_2}}}.$$
Therefore,
$$t^{1-\frac{1}{q}}\|p_t^\varepsilon\|_{L^r(\R^2)}\leq C + t^{1-\frac{1}{q}}\int_0^t\frac{C}{(t-s)^{\frac{3}{2}-\frac{1}{q_1}}s^{\frac{3}{2} - \frac{1}{q_2}}} \ ds.$$
Apply Lemma \ref{lemma:stdComputation} to finish the proof.
\end{proof}
Note that the choice of the constants $A$ and $B$ depends only on the initially chosen $q \in (2,4)$. One may analyze the constants in Condition \eqref{d2exCond} in function of this $q$ to get an optimal condition on $\chi$.
\begin{remark}
Fix $q$ and $\chi$ as in Proposition \ref{prop1}. Gathering all the above estimates one may explicit the constants $B_r(\chi)$ appearing in the Corrolaries \ref{cor:densSmallQ} and \ref{cor:densBigQ} in the following way:
\begin{equation}
\label{dens_estALLq}
B_r(\chi)=
\begin{cases}
& B_q(\chi)^{\frac{(r-1)q}{r(q-1)}} , r< q ;\\
& C_2(r) + C_1(\frac{2r}{r+1}) B_{\frac{qr}{r+1}}(\chi)(\chi \|\nabla c_0\|_{L^2(\R)}C_2(\frac{2r}{r+1})\\
& ~~~~~~~~~~~~~~~~~~~+ B_{\frac{2qr}{(q-2)(r+1)}\wedge q}(\chi)C_2(1\vee \frac{2rq}{3rq+q-4r-2})), r>q.
\end{cases}
\end{equation}
Here, $B_q(\chi)$ is fixed in \eqref{dens_estFIXq}.
\end{remark}
\section{NLMP and Keller-Segel PDE: Existence}
\label{chd2:Thmproof}
\subsection{Non linear martingale problem}
In this section we prove Theorem~\ref{ch:nlsdeD2th1}. First we show the tightness of the constructed probability laws.
\begin{proposition}
\label{prop2}
Let the assumptions of Theorem \ref{ch:nlsdeD2th1} hold. Let $\varepsilon_k:= \frac{1}{k}$, for $k \in \N$.  $\P^{\varepsilon_k}$ denotes the law of the solutions to \eqref{NLSDEd2REG} regularized with $\varepsilon_k$. 
Then, the probability laws $(\P^{\varepsilon_k})_{k\geq 1}$ are tight in $C([0,T];\R^2)$ w.r.t. $k\in \N$.
\end{proposition}
\begin{proof}
For $m>2$ and $0<s<t\leq T$, observe that
\begin{multline*}
\E |X_t^\varepsilon - X_s^\varepsilon|^m\leq \E \left( \left(\int_s^t b^{\varepsilon,1}(u,X_u^\varepsilon)du \right)^{2}+ \left(\int_s^t b^{\varepsilon,2}(u,X_u^\varepsilon)du \right)^{2} \right)^\frac{m}{2}\\ + \E|W_t-W_s|^m.
\end{multline*}
In view of the drift estimate for $r=\infty$ in Corollary \ref{cor:driftEstREG}, one has
\begin{align*}
\E |X_t^\varepsilon - X_s^\varepsilon|^m &\leq \left(2\int_s^t \frac{C(\chi, \|\nabla c_0\|_{L^2(\R^2)} )}{\sqrt{u}} du \right)^{m} + C (t-s)^{\frac{m}{2}}\\
&\leq C(\chi, \|\nabla c_0\|_{L^2(\R^2)} ) (t-s)^{\frac{m}{2}}. 
\end{align*}
Then, Kolmogorov's criterion implies tightness.
\end{proof}
Now we prove two auxiliary lemmas useful for the proof of Theorem \ref{ch:nlsdeD2th1}.
\begin{lemma}
\label{lemma:aux1}
Let $t>0$ and $r\in (2,\infty]$. Then,
\begin{equation*}
\|b_0^{\varepsilon_k}(t,\cdot)-b_0(t,\cdot)\|_{L^r(\R^2)} \to 0, k\to \infty.
\end{equation*}
\begin{proof}
Notice that for $t>0$ and $x \in \R^2$
\begin{align*}
|b_0^{\varepsilon_k}(t,x)-b_0(t,x)|& \leq C \frac{\chi e^{-\lambda t} \varepsilon_k}{t(t+\varepsilon_k)}\left|\int_{\R^2}\nabla c_0(x-y)e^{-\frac{|y|^2}{2t}} \ dy\right|\\
&\leq \frac{\varepsilon_k \|\nabla c_0\|_{L^2(\R^2)}}{\sqrt{t}(t+\varepsilon_k)}.
\end{align*}
Thus, $\|b_0^{\varepsilon_k}(t,\cdot)-b_0(t,\cdot)\|_{L^\infty(\R^2)}\to 0, k\to \infty$.  
Similarly, for $t>0$ and $r>2$,
$$\|b_0^{\varepsilon_k}(t,\cdot)-b_0(t,\cdot)\|_{L^r(\R^2)}\leq \|\nabla c_0\|_{L^2(\R^2)}  \frac{\varepsilon_k}{t(t+\varepsilon_k)}C t^{\frac{1}{r}+\frac{1}{2}}. $$
Let $k\to \infty$ to finish the proof.
\end{proof}
\end{lemma}
\begin{lemma}
Let $t>0$, $1<r<2$ and $i\in \{1,2\}$. Then, for any $s\in (0,t)$, one has
\label{lemma:aux2}
\begin{equation*}
\|K^{\varepsilon_k,i}_{t-s} - K^{i}_{t-s}\|_{L^r(\R^2)} \to 0, k\to \infty.
\end{equation*}
\end{lemma}
\begin{proof}
 Fix $0<s<t$ and $x\in \R^2$. Notice that
$$|K^{\varepsilon_k,i}_{t-s}(x)- K^{i}_{t-s}(x)|\leq \frac{(t-s)\varepsilon_k + \varepsilon_k^2}{(t-s)^2 (t-s+\varepsilon_k)^2} |x^i|e^{\frac{|x|^2}{2(t-s)}}.$$
Thus, for any  $x\in \R^2$, we have that $|K^{\varepsilon_k,i}_{t-s}(x)- K^{i}_{t-s}(x)| \to 0, k\to \infty$. After integration, for any $1<r<2$ one has
$$\|K^{\varepsilon_k,i}_{t-s} - K^{i}_{t-s}\|_{L^r(\R^2)}\leq C_r\frac{(t-s)\varepsilon_k + \varepsilon_k^2}{(t-s)^2 (t-s+\varepsilon_k)^2} (t-s)^{\frac{1}{2}+\frac{1}{r}}. $$
Let $k\to \infty$ to finish the proof.
\end{proof}
\begin{proof}[Proof of Theorem \ref{ch:nlsdeD2th1}]
In view of Proposition \ref{prop2}, there exists a weakly convergent subsequence of 
$(\P^{\varepsilon_k})_{k\geq 1} $, that we will still denote by $(\P^{\varepsilon_k})_{k\geq1} $.
Denote its limit by $\P^\infty$. To prove Theorem~\ref{ch:nlsdeD2th1}, we will prove that $\P^\infty$ solves 
the martingale problem~\hyperref[defMP2d]{$(MP)$}.

Part $i)$ trivially holds.

Now we prove $ii).$ Define the functional $\Lambda_t(\varphi)$ by
 $$\Lambda_t(\varphi):= \int_{\R^2} \varphi(y) \P^\infty_t(dy), \quad \varphi\in C_K(\R^2).$$
By weak convergence we have
$$\Lambda_t(\varphi)= \lim_{k\to \infty} \int \varphi(y)p^{\varepsilon_k}_t(y)dy,$$
and thus for any  $1<r<\infty$ and its conjugate $r'$, in view of Proposition \ref{prop1} and Corollaries \ref{cor:densSmallQ} and \ref{cor:densBigQ} one has
$$|\Lambda_t(\varphi)|\leq\frac{C}{t^{1-\frac{1}{r}}} \|\varphi\|_{L^{r'}(\R^2)}.$$
Therefore, for each $0<t\leq T$, $\Lambda_t$ is a bounded linear functional 
on a dense subset of $L^{r'}(\R^2)$. Thus, $\Lambda_t$ can be extended to a linear 
functional on  $L^{r'}(\R^2)$. By Riesz-representation theorem (e.g. \cite[Thm. 4.11 and 4.14]{Brezis}), there exists 
a unique $p^\infty_t \in L^r(\R^2)$ such that $\|p^\infty_t \|_{L^r(\R^2)} 
\leq\frac{C}{t^{1-\frac{1}{r}}}$ and $p^\infty_t$ is the probability
density of~$\P^\infty_t(dy)$.

It remains to prove $iii)$. Set
\begin{multline*}
M_t^\infty:= f(w_t)-f(w_0)
-\int_0^t \big[ \triangle f(w_u)+\nabla f(w_u)\cdot \Big(b_0(u,w_u)+\\ \chi \int_0^u e^{-\lambda (u-\tau)}\int K_{u- \tau}(w_u-y) p^\infty_\tau(y)\ dy \ d\tau\big)\Big]du.
\end{multline*}
In order to prove that $(M_t^\infty)_{t\leq T}$ is a $\P^\infty$ martingale, we will check that for any $N\geq 1$, $0\leq t_1<\dots<t_N<s \leq t\leq T$ and any $\phi \in C_b((\R^2)^N)$,   one has
\begin{equation}
\label{proofTh1:existence:MP}
\E_{\P^\infty}[(M_t^\infty-M_s^\infty)\phi(w_{t_1}, \dots, w_{t_N})] = 0.
\end{equation}
As $\P^{\varepsilon_k}$ solves the non--linear martingale problem related to \eqref{NLSDEd2REG} with $\varepsilon_k = \frac{1}{k}$, one has
\begin{multline*}
M_t^k:= f(w_t)-f(x(0))- \chi\int_0^t \big[
 \triangle f(w_u)+\nabla f(w_u)\cdot(b^{\varepsilon_k}_0(u,w_u)\\
 + \chi \int_0^u e^{-\lambda (u-\tau)} (K^{\varepsilon_k}_{u-\tau}\ast 
p^{\varepsilon_k}_\tau)(w_u)d\tau\big)]du 
\end{multline*}
is a martingale under $\P^{\varepsilon_k}$. Thus,
\begin{align*}
    & 0=\E_{\P^{\varepsilon_k}}[(M^k_t-M^k_s)\phi(w_{t_1}, \dots, w_{t_N})] = \E_{\P^{\varepsilon_k}}[\phi(\dots)(f(w_t)-f(w_s))] \\
    &+ 
    \E_{\P^{\varepsilon_k}}[\phi(\dots)\int_s^t 
    \triangle f(w_u)du ]+\E_{\P^{\varepsilon_k}}[\phi(\dots)\int_s^t \nabla f(w_u)\cdot
    b_0^{\varepsilon_k}(u,w_u) du ]\\
    &+\chi \ \E_{\P^{\varepsilon_k}}[\phi(\dots)\int_s^t \nabla f(w_u)\cdot  \int_0^u e^{-\lambda (u-\tau)} (K^{\varepsilon_k}_{u-\tau}\ast 
    p^{\varepsilon_k}_\tau)(w_u)d\tau du ].
\end{align*}
Since $(\P^{\varepsilon_k})$ weakly converges to $\P^\infty$, the first two terms on the r.h.s. converge to their analogues in \eqref{proofTh1:existence:MP}. It remains to check the convergence of the last two terms. We will analyze separately the parts coming from the linear and non-linear drifts. 

\paragraph{Linear part} 
Observe that
\begin{align*}
& \E_{\P^{\varepsilon_k}}[\phi(\dots)\int_s^t \nabla f(w_u)\cdot
    b_0^{\varepsilon_k}(u,w_u) du ]- \E_{\P^\infty}[\phi(\dots)\int_s^t \nabla f(w_u)\cdot
    b_0(u,w_u) du ]\\
    & =\Big(\E_{\P^{\varepsilon_k}}[\phi(\dots)\int_s^t \nabla f(w_u)\cdot
   ( b_0^{\varepsilon_k}(u,w_u)- b_0(u,w_u)) du\Big) \\
    &+\Big( \E_{\P^{\varepsilon_k}}[\phi(\dots)\int_s^t \nabla f(w_u)\cdot
    b_0(u,w_u) du ]\\
    & ~~~~~~~~~ - \E_{\P^\infty}[\phi(\dots)\int_s^t \nabla f(w_u)\cdot
    b_0(u,w_u) du ]\Big)=: I_k+II_k.
\end{align*}
We start from $II_k$. Define for $x \in  C([0,T]; \R^2)$ the functional 
$$F(x):= \phi(x_{t_1}, \dots, x_{t_N}) \int_s^t \nabla f(x_u)\cdot b_0(u, x_u)du.$$
In view of Lemma \ref{lemma:linDrift}, for $u>0$ and $i=1,2$, the function $b_0^i(u, \cdot)$ is bounded and continuous on $\R^2$ and one has $\|b_0^i(t,\cdot)\|_{L^\infty(\R^2)}\leq \frac{C}{\sqrt{t}}$. By dominated convergence one gets that $F(\cdot)$ is  continuous. In addition, $F(\cdot)$ is bounded on  $C([0,T]; \R^2)$. Thus, by weak convergence, $II_k \to 0$, as $k \to \infty$.

We turn to $I_k$:
$$|I_k|\leq \|\phi\|_\infty \int_s^t \sum_{i=1}^2 \int_{\R^2}  |\nabla_i f(z)(b_0^{\varepsilon_k, i}(u,z)- b_0^i(u,z)) | p^{\varepsilon_k}_u(z)\ dz \ ds.$$
Apply the H\"{o}lder's inequality for $\frac{1}{\lambda}+\frac{1}{\lambda'}=1$ such that $1<\lambda<2$. In view of Corollary \ref{cor:densSmallQ}, one has
$$|I_k|\leq \|\phi\|_\infty \|\nabla f\|_\infty  \int_s^t \frac{C}{u^{1-\frac{1}{\lambda}}} \sum_{i=1}^2 \| b_0^{\varepsilon_k, i}(u,\cdot)- b_0^i(u,\cdot)\|_{L^{\lambda'}(\R^2)}  du.$$
In view of Lemma \ref{lemma:aux1}, $\| b_0^{\varepsilon_k, i}(u,\cdot)- b_0^i(u,\cdot)\|_{L^{\lambda'}(\R^2)}  \to 0$ as $k \to \infty$. In addition, Lemmas \ref{lemma:linDrift} and \ref{lemma:linDriftREG} lead to
$$\frac{C}{u^{1-\frac{1}{\lambda}}} \sum_{i=1}^2 \| b_0^{\varepsilon_k, i}(u,\cdot)- b_0^i(u,\cdot)\|_{L^{\lambda'}(\R^2)}  \leq \frac{C}{u^{\frac{1}{\lambda'}+\frac{1}{2}- \frac{1}{\lambda'}}}.$$
By dominated convergence, $I_k \to 0$, as $k\to \infty$. 
\paragraph{Non-linear part} 
As in the linear part, we decompose
\begin{align*}
  &\E_{\P^{\varepsilon_k}}[\phi(\dots)\int_s^t \nabla f(w_u)\cdot  \int_0^u (K^{\varepsilon_k}_{u-\tau}\ast 
    p^{\varepsilon_k}_\tau)(w_u)d\tau du ]\\
  &  - \E_{\P^\infty}[\phi(\dots)\int_s^t \nabla f(w_u)\cdot  \int_0^u (K_{u-\tau}\ast 
    p^\infty_\tau)(w_u)d\tau du ]\\
  &   \leq \Big(\E_{\P^{\varepsilon_k}}[\phi(\dots)\int_s^t \nabla f(w_u)\cdot  \int_0^u (K^{\varepsilon_k}_{u-\tau}\ast 
    p^{\varepsilon_k}_\tau)(w_u)d\tau du ]\\
   & - \E_{\P^{\varepsilon_k}}[\phi(\dots)\int_s^t \nabla f(w_u)\cdot  \int_0^u (K_{u-\tau}\ast 
    p^{\infty}_\tau)(w_u)d\tau du ]\Big)\\ 
    &+\Big(\E_{\P^{\varepsilon_k}}[\phi(\dots)\int_s^t \nabla f(w_u)\cdot  \int_0^u (K_{u-\tau}\ast 
    p^{\infty}_\tau)(w_u)d\tau du ]\\
     &- \E_{\P^\infty}[\phi(\dots)\int_s^t \nabla f(w_u)\cdot  \int_0^u (K_{u-\tau}\ast 
    p^\infty_\tau)(w_u)d\tau du ]\Big)\\
    & =: C_k + D_k.
\end{align*}
Start from $D_k$. Similarly to the linear part, we need the boundness and continuity of the functional
$$H(x):= \phi(x(t_1), \dots, x(t_N))\int_s^t \nabla f(x(u))\cdot  \int_0^u (K_{u-\tau}\ast 
    p^{\infty}_\tau)(x(u))d\tau du, $$
    where $ x\in \Cc([0,T];\R^2).$ 
The continuity comes from the fact that the kernel is a continuous function on $\R^2$ whenever $\tau<u$. Namely, if $x_n \in \Cc([0,T];\R^2)$ converges to $x\in \Cc([0,T];\R^2)$, then $K^i_{u-\tau}(x_n(u)-y) \to K^i_{u-\tau}(x(u)-y)$. In addition $|K^i_{u-\tau}(x_n(u)-y)p^\infty_\tau(y)|\leq \frac{C}{(u-\tau)^{3/2}}p^\infty_\tau(y)$, for $i \in \{1,2\}$, as $n \to \infty$. Thus, by dominated convergence, for  $\tau<u$ one has
$$ K^i_{u-\tau} \ast p^\infty_\tau(x_n(u))\to K^i_{u-\tau} \ast p^\infty_\tau(x(u)), \ n \to \infty.$$
For $\frac{1}{r}+ \frac{1}{r'}=1$ such that $r >2$ apply  H\"{o}lder's inequality and after the estimate in $ii)$. It comes
$$|K^i_{u-\tau} \ast p^\infty_\tau(x_n(u))|\leq \frac{C_r}{(u-\tau)^{\frac{3}{2}-\frac{1}{r'}}\tau^{1-\frac{1}{r}}}.$$

By dominated convergence,
$$\int_0^u (K^i_{u-\tau}\ast 
    p^\infty_\tau)(x_n(u))d\tau \to \int_0^u (K^i_{u-\tau}\ast 
    p^\infty_\tau)(x(u))d\tau, n \to \infty.  $$
Moreover, in view of Lemma \ref{lemma:stdComputation}, one has
$$\left| \nabla f(x_n(u))\cdot \int_0^uK_{u-\tau} \ast p^\infty_\tau(x_n(u))d \tau\right | \leq C \| \nabla f\|_\infty  \frac{\beta(1-\frac{1}{r},\frac{3}{2}-\frac{1}{r'})}{\sqrt{u}}.$$
Finally, after one more application of dominated convergence the continuity of the functional $H$ follows. This procedure obviously implies $H$ is a bounded functional on $\Cc([0,T];\R^2)$. Thus, by weak convergence, $D_k$ converges to zero.

 We turn to $C_k$. Let us just for this part denote by $b^i(t,x):=\int_0^t (K^{i}_{t-s}\ast p^\infty_s)(x) ds$ and  $b^{k,i}(t,x):=\int_0^t (K^{\varepsilon_k,i}_{t-s}\ast p^{\varepsilon_k}_s)(x)$. Assume for a moment that for any $t>0$ and $x\in \R^2$, one has
\begin{equation}
\label{driftConv}
\lim_{k\to \infty}\left|b^{k,i}(t,x)- b^i(t,x) \right|=0.
\end{equation} 
  Notice that 
 $$|C_k|\leq \|\phi\|_\infty \int_s^t \sum_{i=1}^2 \int_{\R^2}|\nabla_i f(z) (b^{k,i}(u,z)- b^i(u,z))|p^{\varepsilon_k}_u(z)dz.$$
 After H\"{o}lder inequality for $\frac{1}{r}+ \frac{1}{r'}=1$ such that $r>2$, one has 
 $$|C_k|\leq \|\phi\|_\infty \int_s^t \frac{C}{u^{1-\frac{1}{r'}}}  \sum_{i=1}^2 \left(\int |\nabla_i f(z)|^r |b^{k,i}(u,z)- b^i(u,z)|^r dz  \right)^{1/r} du.$$
Let $u>0$. In view of \eqref{driftConv}, $|b^{k,i}(u,z)- b^i(u,z)|^r \to 0$ as $k \to \infty$. Now, we do not omit $|\nabla_i f(z)|^q$ as in the linear part. Instead, we use it in order to integrate in space with respect to drift bounds. Namely, for $u>0$ and $i=1,2$, we have that $|b^{k, i}(u,\cdot)|+|b^i(u,\cdot)|\leq \frac{C}{\sqrt{u}}$. Thus,
 $$|\nabla_i f(z)|^r |b^{k,i}(u,z)- b^i(u,z)|^r \leq \frac{C}{u^{\frac{r}{2}}}|\nabla_i f(z)|^r.$$
 By dominated convergence, 
 $$\|\nabla_i f(\cdot)(b^{k,i}(u,\cdot)- b^i(u,\cdot)) \|_{L^r(\R^2)} \to 0, \  k \to \infty.$$
  Using that $\|b^{k,i}(u,\cdot)\|_{L^r(\R^2)}+\| b^i(u,\cdot)) \|_{L^r(\R^2)}\leq \frac{C}{u^{\frac{1}{2}-\frac{1}{r}}}$, one gets
 $$\frac{1}{u^{1-\frac{1}{r'}}} \|\nabla_i f(\cdot)(b^{k,i}(u,\cdot)- b^i(u,\cdot)) \|_{L^r(\R^2)}\leq  \|\nabla_i f(\cdot)\|_\infty  \frac{C}{u^{\frac{1}{2}-\frac{1}{r} + 1-\frac{1}{r'}}}$$
Thus, by dominated convergence, we get that $C_k\to 0$, as $k\to \infty$. 
 
As all the terms converge, we get that \eqref{proofTh1:existence:MP} holds true. Thus, the process $(M_t^\infty)_{t\leq T}$ is a $\P^\infty$ martingale.

To finish the proof, it remains to show \eqref{driftConv}.
For $t>0$, $x\in \R^2$ and $i \in {1,2}$, one has
 \begin{align*}
& |b^{k,i}(t,x)- b^i(t,x) |=\left|\int_0^t (K^{\varepsilon_k,i}_{t-s}\ast p^{\varepsilon_k}_s)(x) ds- \int_0^t (K^{i}_{t-s}\ast p^\infty_s)(x) ds \right|\\
&\leq \left|\int_0^t ((K^{\varepsilon_k,i}_{t-s}-K^{i}_{t-s})\ast p^{\varepsilon_k}_s)(x)\ ds\right| + \left|\int_0^t (K^{i}_{t-s}\ast ( p^{\varepsilon_k}_s-p^\infty_s))(x) ds \right|\\
&=: A_k + B_k.
\end{align*}
We start from $B_k$. For $s<t$ and $i=1,2$, the kernel $K^i_{t-s}(\cdot)$ is a continuous and bounded function on $\R^2$. Thus, by weak convergence we have that $\lim_{k \to \infty}(K^{i}_{t-s}\ast p^{\varepsilon_k}_s)(x)= (K^{i}_{t-s}\ast p^\infty_s)(x)$. In addition, for $r>2$ H\"{o}lder's inequality, part $ii)$ and Proposition \ref{prop1} lead to
$$|(K^{i}_{t-s}\ast p^{\varepsilon_k}_s)(x)- (K^{i}_{t-s}\ast p^\infty_s)(x) |\leq \frac{C_r}{(t-s)^{\frac{3}{2}- \frac{1}{r'} } s^{1-\frac{1}{r}} }.$$
As the bound is integrable in $(0,t)$, the dominated convergence theorem implies that $B_k\to 0$, as $k \to \infty$.

In $A_k$ we apply the H\"{o}lder's inequality with $1<r<2$ and the density bounds from Corollary \ref{cor:densSmallQ}. It comes
$$|A_k|\leq \int_0^t \|K^{\varepsilon_k,i}_{t-s} - K^{i}_{t-s}\|_{L^{r}(\R^2)}\frac{C_r}{s^{1-\frac{1}{r'}}}\ ds.$$
In view of \eqref{kerneld2Lp} and \eqref{reg2dKernelEst}, one has
\begin{equation*}
\|K^{\varepsilon_k,i}_{t-s} - K^{i}_{t-s}\|_{L^r(\R^2)}\leq \|K^{\varepsilon_k,i}_{t-s}\|_{L^r(\R^2)} + \|K^{i}_{t-s}\|_{L^r(\R^2)}\leq \frac{C_r}{(t-s)^{\frac{3}{2}-\frac{1}{r}}}.
\end{equation*}
Lemma \ref{lemma:aux2} and the preceding inequality enable us to 
apply dominated convergence. Thus, $A_k\to 0$, as $k \to \infty$.
\end{proof}
\subsection{Keller-Segel PDE}
\label{chd2:KSex}
In this section we prove Theorem \ref{existenceKSd2}. 
 
Fix $\chi>0$ as in \eqref{d2exCond}. Denote by $\rho(t,\cdot)\equiv q_t(x)$ the time marginals of 
the solution to \hyperref[defMP2d]{$(MP)$} constructed in Theorem \ref{ch:nlsdeD2th1}. As such, $\rho$ satisfies for any $1\leq q <\infty$,
$$\sup_{t\leq T} t^{1-\frac{1}{q}}\|\rho(t,\cdot)\|_{L^q(\R^2)}\leq B_q(\chi).$$
Here $B_q(\chi)$ are, depending on $q$, given in either \eqref{dens_estFIXq} or \eqref{dens_estALLq}. The corresponding drift function satisfies for any $1\leq r \leq \infty$,
$$t^{\frac{1}{2}-\frac{1}{r}} \|b(t,\cdot;\rho)\|_{L^r(\R^2)}\leq C_r(\chi).$$
 Following the arguments in Proposition 4.1 in \cite{Mi-De} one may derive the mild equation for $\rho(t,\cdot)$. The above estimates ensure that everything is well defined. One arrives to the following: For any $f\in C_K^\infty(\R^2)$ and any $t\in(0,T]$,
 \begin{align*}
 &\int f(y)\rho(t,y)\,dy = \int f(y) (g_t\ast \rho_0)(y) dy\\
 & -   \sum_{i=1}^2 \int f(y) 
\int_0^t[  \nabla_i g_{t-s} \ast
(b^i(s,\cdot;\rho)\rho(s,\cdot))](y) \ ds \  dy.
 \end{align*}
Thus $\rho$ satisfies in the sense of the distributions
\begin{equation}
\label{d2mildeq}
\rho(t,\cdot)= g_t\ast \rho_0 - \sum_{i=1}^2 \int_0^t \nabla_i g_{t-s} \ast(
b^i(s,\cdot;\rho))\rho(s,\cdot)) \,ds.
\end{equation}
Now, define the function $c(t,x)$ as 
$$ c(t,x) :=e^{-\lambda t}(g(t,\cdot)\ast c_0)(x) + \int_0^t 
e^{-\lambda s} \rho(t-s,\cdot)\ast g(s,\cdot)(x)~ds. $$
Thanks to the density estimates $c(t,x)$ is well defined for all $x\in \R^2$ as soon as $t>0$. Indeed, 
\begin{align*}
|c(t,x)|&\leq  \frac{\|c_0\|_{L^2(\R^2)}}{\sqrt{t}} +  C\int_0^t \|\rho(t-s,\cdot)\|_{L^2(\R^2)} \|g_s\|_{L^2(\R^2)}  \,ds \\
&\leq \frac{\|c_0\|_{L^2(\R^2)}}{\sqrt{t}} +  C\beta(\frac{1}{2},\frac{1}{2}). 
\end{align*}
It is obvious that $c(t,\cdot) \in L^2(\R^2)$. Thanks to the density estimates and the fact that $g_t$ is strongly derivable as soon as $t>0$, $c(t,x)$ is derivable in any point $x$ and
$$\frac{\partial}{\partial x_i}c(t,x)=e^{-\lambda t} \nabla_i ( g(t,\cdot)\ast c_0)(x) + \int_0^t 
e^{-\lambda s} (\rho_{t-s}\ast \nabla_i g(s,\cdot))(x)~ds.$$
 The fact that $c_0 \in H^1(\R^2)$ enables us to write $\nabla_i (g(t,\cdot)\ast c_0) =(g(t,\cdot)\ast \nabla_ic_0) $. Now, remark that $\chi\frac{\partial}{\partial x_i}c(t,x)$ is exactly the drift in \eqref{d2mildeq}. Thus, the couple $(\rho,c)$ satisfies Definition \ref{notionOfSold2}. 
 
 The following remark will be useful in the proof of uniqueness of the above constructed solution:
 \begin{remark}
 \label{rem:foruniq}
 Take the same $\chi>0$ as above. Let $(\tilde{\rho},\tilde{c})$ be a solution to the Keller-Segel equation in the sense of Definition \ref{notionOfSold2} with such $\chi$. Let $\tilde{C}_q(\chi)$ be the corresponding constant in \eqref{rho_space}. Applying step by step the same computations as in Proposition \ref{prop1} to the Eq. \eqref{eq:rho-KSd2} and afterwards all the computations from Corollaries \ref{cor:densSmallQ} and \ref{cor:densBigQ}, one obtains that for any $q>1$, $\tilde{C}_q(\chi)$ can be chosen equal to $B_q(\chi)$ fixed in \eqref{dens_estFIXq} and \eqref{dens_estALLq}.
 \end{remark}
\section{NLMP and Keller-Segel PDE : Uniqueness}
\label{sec:uniq}
\subsection{Keller-Segel PDE}
In principle, uniqueness of solutions to the Keller-Segel system should be derived from the stability theorem 2.6 in \cite{Corrias2014}. We believe the statement is true. However, at the beginning of the proof a term seems to be missing in the expression for the difference of two integral solutions at time $t+\tau$.  The missing term seems to jeopardise the Gronwall lemma used later on. We propose here a proof for uniqueness only which is not a stability result and, thus, does not require the use of the Gronwall lemma. The price to pay is an additional condition on the size of parameter~$\chi$. 
\begin{proof}[Proof of Theorem \ref{th:uniqKS}]
By assumption, $q$ is fixed such that $q\in (2,4)$.
 Assume there exist two pairs $(\rho^i,c^i)$, $i=1,2$, satisfying Definition \ref{notionOfSold2} with the same initial condition $(\rho_0, c_0)$. As such, one has 
$$\forall ~ 1\leq r<\infty~ \exists ~ C^i_r(\chi)> 0:~~~ \sup_{t\leq T}t^{1-\frac{1}{r}}\|\rho^i_t\|_{L^r(\R^2)}\leq C^i_r(\chi).$$
By Remark \ref{rem:foruniq}, one may assume that $C^1_r(\chi)=C^2_r(\chi)=B_r(\chi)$ given by \eqref{dens_estFIXq} for $r=q$ and by \eqref{dens_estALLq} otherwise. Then, after H\"older's inequality, one has
$$ \sup_{t\leq T}\sqrt{t}\|\nabla c^i_t\|_{L^\infty(\R^2)} \leq C ( \|\nabla c_0\|_{L^2(\R^2)}+B_q(\chi)).$$
To get uniqueness it suffices to show that for an $r>1$
$$\sup_{t\leq T} t^{1-\frac{1}{r}}\|\rho^1_t-\rho_t^2\|_{L^{r}(\R^2)}= 0.$$
Let us fix an $r$ such that $\frac{1}{r}+ \frac{1}{q}<1$. As $q\in (2,4)$, one has $r \in (1,2)$.  Denote $$f(T):=\sup_{t\leq T} t^{1-\frac{1}{r}}\|\rho^1_t-\rho_t^2\|_{L^{r}(\R^2)}.$$ From Definition~\ref{notionOfSold2}, one has
\begin{equation}
\label{proof:uniq:eq1}
\begin{split}
&t^{1-\frac{1}{r}}\|\rho^1_t-\rho_t^2\|_{L^{r}(\R^2)}\\ 
&\leq  \chi\sum_{i=1}^2\int_0^t \|\nabla_i g_{t-s} \ast (\nabla_i c^1(s,\cdot)~(\rho^1(s,\cdot)-\rho^2(s,\cdot))\|_{L^{r}(\R^2)}ds\\ 
&+ \chi \sum_{i=1}^2 \int_0^t \|\nabla_i g_{t-s} \ast (\rho^2(s,\cdot)~(\nabla_i c^1(s,\cdot)-\nabla_i c^2(s,\cdot))\|_{L^{r}(\R^2)}ds\\ &=: I+II.
\end{split}
\end{equation}
Apply the above estimate on $\nabla c^1_s$, then Convolution inequality \eqref{convIneq}, Eq.~\eqref{kerneld2Lp} and Lemma \ref{lemma:stdComputation}. It comes
\begin{multline*}
I\leq C  \chi ( \|\nabla c_0\|_{L^2(\R^2)}+B_q(\chi)) f(T)t^{1-\frac{1}{r}}\int_0^t\frac{\|\nabla_i g_{t-s} \|_{L^1(\R)}}{s^{3/2-1/r}}  \ ds\\ \leq C  \chi ( \|\nabla c_0\|_{L^2(\R^2)}+B_q(\chi))f(T) .
\end{multline*}
Then, apply successively Convolution inequality \eqref{convIneq} and H\"older's inequality for $q$ and its conjugate $q'$ to get
\begin{multline}
\label{proof:uniq:eq2}
II
 \leq t^{1-\frac{1}{r}} \chi\\ \times \sum_{i=1}^2 \int_0^t \|\nabla_i g_{t-s}\|_{L^r(\R^2)}\|\rho^2(s,\cdot)\|_{L^q(\R^2)} \| (\nabla_i c^1(s,\cdot)-\nabla_i c^2(s,\cdot))\|_{L^{q'}(\R^2)} \ ds .
\end{multline}
Now, let $x$ be such that $1+\frac{1}{q'}-\frac{1}{r}=\frac{1}{x}$. Remark that $1<x<2$, as we supposed $0<\frac{1}{q}+\frac{1}{r}<1$. In view of Convolution Inequality \eqref{convIneq}, Eq. \eqref{kerneld2Lp} and Lemma \ref{lemma:stdComputation}, one has
\begin{multline*}
\|\nabla_i c^1(t,\cdot)-\nabla_i c^2(t,\cdot)\|_{L^{q'}(\R)} \leq \int_0^t \|K_{t-s}\|_{L^{x}(\R^2)}\|\rho^1_s-\rho_s^2\|_{L^{r}(\R^2)} \ ds\\\leq f(T)\int_0^t \frac{C}{{(t-s)^{\frac{3}{2}-\frac{1}{x}}}s^{1-\frac{1}{r}}} ds =  \frac{C f(T)}{t^{\frac{1}{2}-\frac{1}{q'}}}.
\end{multline*}
In view of \eqref{kerneld2Lp} and the preceding estimate in \eqref{proof:uniq:eq2}, one has
$$II\leq t^{1-\frac{1}{r}}\chi C f(T)B_q(\chi) \int_0^s \frac{1}{(t-s)^{\frac{3}{2}-\frac{1}{r}}s^{1-\frac{1}{q}}s^{\frac{1}{2}-(1-\frac{1}{q})}} \ ds = \chi C f(T)B_q(\chi).$$
Use the estimates on $I$ and $II$ in \eqref{proof:uniq:eq1} and then take the $\sup_{t\leq T}$. It comes
$$f(T)\leq C_0\chi(\|\nabla c_0\|_{L^2(\R^2)}+ B_q(\chi))f(T).$$
Thus, if $C_0\chi(\|\nabla c_0\|_{L^2(\R^2)}+ B_q(\chi))<1$, one has the desired result. It remains to notice that $\chi B_q(\chi) \to 0$ as $\chi \to 0$. Thus, the conditions \eqref{d2exCond} and \eqref{chiCondUniq} are compatible.
\end{proof}
\subsection{From linearised martingale problem to NLMP}
The goal of this section is to prove Theorem \ref{UniquenessLawD2}. Let us fix $\chi>0$ that satisfies conditions \eqref{d2exCond} and \eqref{chiCondUniq}.

Theorem \ref{th:uniqKS} tells us that time marginal densities of a solution to~\hyperref[defMP2d]{$(MP)$} uniquely solve the mild equation~\eqref{eq:rho-KSd2}. Let us denote these uniquely determined time marginals with $(\rho_s)_{s\geq 0}$.
 Then, the standard argument to get uniqueness of a solution to \hyperref[defMP2d]{$(MP)$} is to uniqueness for the linearised version of~\hyperref[defMP2d]{$(MP)$}.

We define the linearised process
\begin{equation}
\label{linProcd2}
\begin{cases}
& d\tilde{X}_t= b_0(t, \tilde{X}_t) dt + \chi \int_0^t (K_{t-s}\ast \rho_s) (\tilde{X}_t) \ ds \ dt + dW_t,\\
& \tilde{X}_0 \sim \rho_0,
\end{cases}
\end{equation}
where $b_0$ is as in \eqref{def:b0andK}.
We will denote in this section 
$$b(t,x):=b_0(t, x) dt + \chi \int_0^t K_{t-s}\ast \rho_s (x) \ ds.$$
Having in mind the properties of $(\rho_s)_{s\geq 0}$, one has
\begin{equation}
\label{LrNormB}
\forall r \in [2,\infty] ~ \exists C_r: ~~~ \sup_{t\leq T} t^{\frac{1}{2}-\frac{1}{r}}\|b(t,\cdot)\|_{L^r(\R^2)}\leq C_r.
\end{equation}

As the drift of \eqref{linProcd2} is neither uniformly bounded in time and space, nor it is in the framework of Krylov and R\"ockner, there is no immediate result that gives us the uniqueness in law for \eqref{linProcd2}.

To get this uniqueness result we will use the so called transfer of uniqueness from a (linear) Fokker-Planck equation to the corresponding martingale problem (see \cite{Trevisan} and the references therein).

Thus, we define the (linear) martingale problem related to \eqref{linProcd2} starting from any time $0\leq s < T$ with initial probability density function on $\R^2$ that we denote by $q_s$.
\begin{mydef}
\label{defMP2dLIN}
Let $T>0$, $\chi>0$ and $0\leq s <T$. Consider the canonical space $\mathcal{C}([s,T];\R^2)$ equipped with its canonical filtration. Let $\Q$ be a probability measure on this canonical space and denote by $\Q_t$ its one dimensional time marginals. $\Q$ solves the non-linear 
martingale problem \hyperref[defMP2dLIN]{$(LMP)$} if:
\begin{enumerate}[(i)]
\item $\Q_s$ admits a probability density $q_s$.
\item For any $t\in (s,T]$, $\Q_t$ have densities $q_t$ w.r.t. Lebesgue measure on $\R$. In addition, they satisfy  $$\forall r\in (1,\infty)\ \exists C_r(\chi)>0: \quad \sup_{t\in (0,T)}(t-s)^{1-\frac{1}{r}}\|q_t\|_{ L^r(\R^2)}\leq C_r(\chi).$$
\item For any $f \in C_K^2(\R^2)$ the process $(M_t)_{s\leq t\leq T}$, 
defined as
$$M_t:=f(w_t)-f(w_s)-\int_s^t \big[\frac{1}{2} \triangle f(w_u)+\nabla f(w_u)\cdot b(u,w_u)]du $$
is a $\Q$-martingale where $(w_t)$ is the canonical process.
\end{enumerate}
\end{mydef} 
It is clear that any solution to \hyperref[defMP2d]{$(MP)$} is a solution to \hyperref[defMP2dLIN]{$(LMP)$} with $s=0$ and $q_0=\rho_0$.



To prove the uniqueness of solution to $(LMP)$ with $s=0$ and $q_0=\rho_0$, the goal is to use Lemma 2.12 in~\cite{Trevisan} in the sense $i)$ implies $ii)$ for $s=0$. This result is stated in the sequel once all the objects appearing in it are introduced.

Firstly, one derives in the usual way the following mild equation satisfied by the laws $(\tilde{p}_t)_{t\leq T}$ in the sense of the distributions:
\begin{equation}
\label{linMildEQd2}
\tilde{p}_t=g_t \ast \rho_0 -\sum_{i=1}^2 \int_0^t \nabla_ig_{t-s} \ast (b(s,\cdot)\tilde{p}_s) \,ds. 
\end{equation}
Now, we define the space $\mathcal{R}_{[0,T]}$ as follows
\begin{equation*}
\mathcal{R}_{[0,T]}:= \{ (\nu_t)_{t\leq T}: \begin{cases}
& 1. ~ \nu_0= \rho_0;\\
& 2.~ \nu_t \text{ is a density function; }\\
& 3. ~ \forall 1<q<\infty, ~\forall 0<t\leq T: ~~ t^{1-\frac{1}{q}}\|\nu_t\|_{L^q(\R^2)}<\infty; \\
& 4.~ \forall 0<t\leq T: \nu_t \text{ satisfies \eqref{linMildEQd2}}.
\end{cases}
\end{equation*}
Repeating the same arguments as in the proof of Theorem \ref{th:uniqKS}, one has the following lemma:
\begin{lemma}
\label{lemma:linmildeqUniq}
Let the assumptions of Theorem \ref{th:uniqKS} hold. Then, Equation~\eqref{linMildEQd2} admits a unique solution in the space $\mathcal{R}_{[0,T]}$.
\end{lemma}

Now, note that for $0<s\leq t \leq T$ one has
\begin{multline*}
\tilde{p}_t=g_{t-s}\ast (g_s \ast \rho_0) -\sum_{i=1}^2 \int_0^s g_{t-s}\ast(\nabla_i g_{s-u} \ast (b(u,\cdot)\tilde{p}_u) )\,du \\- \sum_{i=1}^2 \int_s^t \nabla_ig_{t-u} \ast (b(u,\cdot)\tilde{p}_u) \,du. 
\end{multline*}
Therefore
\begin{equation}
\label{restartLinMildD2}
\tilde{p}_t=g_{t-s}\ast \tilde{p}_s - \sum_{i=1}^2 \int_s^t \nabla_ig_{t-u} \ast (b(u,\cdot)\tilde{p}_u) \,du.
\end{equation}
From here, for a $\nu \in \mathcal{R}_{[0,s]}$ we  define
\begin{equation}
\label{flow}
p_{s,t}^\nu= g_{t-s}\ast \nu_s - \sum_{i=1}^2 \int_s^t \nabla_ig_{t-u} \ast (b(u,\cdot)p_{s,u}^\nu) \,du.
\end{equation}
Now we define for any $0\leq s< T$ the space
\begin{equation*}
\mathcal{R}_{[s,T]}:= \{ (p_{s,t}^\nu)_{s\leq t\leq T}: \begin{cases}
& 1. ~ \nu \in \mathcal{R}_{[0,s]};\\
& 2. ~ \forall 0\leq t\leq T:  p_{s,t}^\nu \text{ is a  density function; }\\
& 3. ~\forall 1<q<\infty, ~\forall s\leq t\leq T: \\
&~~~~~~~~~~~~(t-s)^{1-\frac{1}{q}}\|p_{s,t}^\nu\|_{L^q(\R^2)}< \infty; \\
&4. ~\forall s\leq t\leq T: p_{s,t}^\nu \text{ satisfies \eqref{flow}}.
\end{cases}
\end{equation*}
In order to prove that two solutions to a martingale problem coincide, the idea in \cite{Trevisan} is to prove by induction that their finite dimensional marginals coincide. As, a priori, in \cite{Trevisan} one does not have the Markov's property for a solution of a martingale problem, the following properties of the family $(\mathcal{R}_{[s,T]})_{0\leq s \leq T}$ are needed to pass from $k$-dimensional marginals to $k+1$-dimensional marginals in the inductive procedure: 
\begin{lemma}
\label{Trevisan_properties}
For any $0\leq s \leq T$, the following two properties are satisfied:
\begin{description}
\item[Property 1] Let $(p_{s,t}^\nu)_{s\leq t\leq T} \in \mathcal{R}_{[s,T]}$ and let $(q_{s,t}^\nu)_{s\leq t\leq T}$ be a family of probability measures that satisfies \eqref{flow} and is such that $q_{s,t}^\nu\leq C p_{s,t}^\nu $ for $t\in [s,T]$. Then, $(q_{s,t}^\nu)_{s\leq t\leq T} \in \mathcal{R}_{[s,T]}$.
\item[Property 2] Let $r\leq s$ and $(q_{r,t}^\nu)_{r\leq t\leq T} \in \mathcal{R}_{[r,T]}$. Then, the restriction\\  $(q_{r,t}^\nu)_{s\leq t\leq T}$ belongs to $\mathcal{R}_{[s,T]}$.
\end{description}
\end{lemma}
\begin{proof}
\textbf{Property 1} Let $s\in[0,T]$, $(p_{s,t}^\nu)_{s\leq t\leq T} \in \mathcal{R}_{[s,T]}$ and let\\ $(q_{s,t}^\nu)_{s\leq t\leq T}$ be a family of probability measures that satisfies \eqref{flow} and is such that $q_{s,t}^\nu \leq C p_{s,t}^\nu $ for $t\in [s,T]$. We should prove that $(q_{s,t}^\nu)_{s\leq t\leq T} \in \mathcal{R}_{[s,T]}$. As for $t\in [s,T],$ we have $q_{s,t}^\nu\leq  C p_{s,t}^\nu $  then for a test function $f\in C_K(\R^2)$ one has
$$|\int f(x) q_{s,t}^\nu (dx)| \leq C  |\int f(x) p_{s,t}^\nu (x)dx|.$$
Let $q>1$ and $q'>1$ such that $\frac{1}{q}+\frac{1}{q'}=1$. As $(p_{s,t}^\nu)_{s\leq t\leq T} \in \mathcal{R}_{[s,T]}$, one has
$$|\int f(x) q_{s,t}^\nu (dx)| \leq C \|f\|_{L^{q'}(\R^2)}\|p_{s,t}^\nu\|_{L^q(\R^2)}\leq \frac{C}{(t-s)^{1-\frac{1}{q}}} \|f\|_{L^{q'}(\R^2)}.$$
By Riesz representation theorem, $q_{s,t}^\nu$ is absolutely continuous with respect to Lebesgue's measure. We still denote its probability density by $q_{s,t}^\nu$ and conclude
$$\|q_{s,t}^\nu\|_{L^q(\R^2)}\leq \frac{C}{(t-s)^{1-\frac{1}{q}}}.$$ 
 Therefore, $(q_{s,t}^\nu)_{s\leq t\leq T} \in \mathcal{R}_{[s,T]}$.

 \textbf{Property 2} Let $r\leq s$ and $(q_{r,t}^\nu)_{r\leq t\leq T} \in \mathcal{R}_{[r,T]}$. We should prove that the restriction~$(q_{r,t}^\nu)_{s\leq t\leq T}$ belongs to $\mathcal{R}_{[s,T]}$. Let $t\geq s$. Notice that
\begin{multline*}
q_{r,t}^\nu = g_{t-s}\ast (g_{s-r} \ast \nu_r) -\sum_{i=1}^2 g_{t-s}\ast \int_r^s \nabla_i g_{s-u} \ast (b(u,\cdot)q_{r,t}^\nu  ) \ du \\ -\sum_{i=1}^2 \int_s^t \nabla_i g_{t-u}\ast(b(u,\cdot)q_{r,t}^\nu) \ du.
\end{multline*} 
 Therefore, for $t\in[s,T]$ one has
  $$q_{r,t}^\nu = g_{t-s}\ast q_{r,s}^\nu  -\sum_{i=1}^2 \int_s^t \nabla_i g_{t-u}\ast(b(u,\cdot)q_{r,t}^\nu) \ du.$$
  In addition, for $t\in[s,T]$ and $r\leq s$, one has 
  $$(t-s)^{1-\frac{1}{m}}\|q_{r,t}^\nu\|_{L^m(\R^2)}\leq (t-r)^{1-\frac{1}{m}}\|q_{r,t}^\nu\|_{L^m(\R^2)}\leq C.$$ Thus the restriction $(q_{r,t}^\nu)_{s\leq t\leq T}$ belongs to $\mathcal{R}_{[s,T]}$.
\end{proof}
 We are ready to state the result \cite[Lemma 2.12]{Trevisan} in our framework:
 \begin{lemma}
 As $\mathcal{R}:= (\mathcal{R}_{[s,T]})_{0\leq s \leq T}$ satisfies the properties in Lemma~\ref{Trevisan_properties}, the following conditions are equivalent:
 \begin{enumerate}[i)]
\item for every $s\in [0,T]$ and $\bar{\nu} \in \mathcal{R}_{[0,s]}$, there exists at most one $\nu \in \mathcal{R}_{[s,T]}$ with $\nu_s= \bar{\nu}_s$.
\item  for every $s\in [0,T]$, if $\Q^1$ and $\Q^2$ to solutions to $(LMP)$ starting from $s$ with $\Q^1_s=\Q^2_s$, then $\Q^1=\Q^2$.
 \end{enumerate}
 \end{lemma}
 To apply the preceding lemma in the sense $i)$ implies $ii) \text{ for } s=0$, it remains to check that for a fixed $\nu \in \mathcal{R}_{[0,s]}$ the equation \eqref{flow} admits a unique solution in  $\mathcal{R}_{[s,T]}$. In order to do so, repeat the same as in the proof of Lemma \ref{lemma:linmildeqUniq} to get the uniqueness of \eqref{flow}. As the constants do not depend on $t, T$, one gets the same condition on $\chi$ for the uniqueness.
 We, thus, conclude the uniqueness of solutions to \hyperref[defMP2dLIN]{$(LMP)$} starting in $s=0$ from $\rho_0$. Then, we conclude the uniqueness of solutions to \hyperref[defMP2d]{$(MP)$}.

\section{Appendix }\label{sec:1chsmooth}
Let $T >0$. On a filtered probability space $(\Omega, \F, \P, (\F_t) )$ equipped with a $d$--dimensional Brownian motion $(W)$ and an $\F_0-$measurable random variable $X_0$, we study the stochastic equation
\begin{equation}
\label{genRegNLSDE0}
\begin{cases}
&dX_t= dW_t  + \Big\{\int_0^t \int_{\R^d} L(t-s, X_t-y)\Q_s(dy) \ ds\Big\} dt , \quad t\leq T,\\
& \Q_s :=  \mathcal{L}( X_{s}),\quad X_0 \sim q_0,
\end{cases}
\end{equation}
where $L$ maps $ [0,T] \times \R^d$ to $\R^d$. In this section we show how to adapt the Proof of Theorem~$1.1$ in \cite{Sznitman} in the framework of the additional time interaction in \eqref{genRegNLSDE0}.

Firstly, the assumption about the regularity of the interaction kernel in~\cite{Sznitman} needs to be replaced by the following hypothesis on the interaction $L$:
\begin{hyp}[H0]
The function  $ L: [0,T] \times \R^d \to \R^d$ satisfies 
\begin{align*}
& \forall (t,x)\in (0,T)\times \R^d, ~~~   |L(t,x)|\leq h_1(t),\\
&  \forall (t,x,y)\in (0,T)\times \R^d \times \R^d, ~~~   |L(t,x) - L(t,y)| \leq h_2(t) |x-y|,
\end{align*}
where $h_i: (0,T) \to \R^+$ is such that there exists $D_T>0$ such that for any $t\leq T$, one has $\int_0^t h_i(s)ds \leq D_T$.
\end{hyp}
Note that the time interaction induces a slight change in (H0) with respect to what is assumed on the interaction kernel in \cite{Sznitman}. We still assume the kernel is bounded and Lipshitz in space, but in order to treat the additional integral in time, we introduce the functions $h_1$ and $h_2$.
 
Let $\Cc:= C((0,T); \R^d)$ be a set of continuous $\R^d$-valued functions defined on $(0,T)$ and $\Pp_T$ be the set of probability measures on $\Cc$.
  For a $\Q\in \Pp_T$ and $(t,x)\in (0,T)\times \R^d$ denote by 
$$b(t,x;(\Q_s)_{s\leq t}):= \int_0^t\int_{\R^d} L(t-s, x-y)\Q_s(dy)~ds.$$
In view of Hypothesis (H0), for a given $\Q\in\Pp_T$ and any $(t,x,y)\in (0,T)\times \R^d \times \R^d$ one has that 
\begin{equation}
\label{smoothLsigma}
\begin{cases}  |b(t,x;(\Q_s)_{s\leq t})|\leq D_T, \\
 |b(t,x;(\Q_s)_{s\leq t})-b(t,y;(\Q_s)_{s\leq t})|\leq D_T|x-y|.
 \end{cases}
\end{equation}
\begin{theorem}
\label{th:smoothInter}
Under the hypothesis (H0), Equation \eqref{genRegNLSDE0} admits a unique strong solution.
\end{theorem}

Let us show how the calculations in \cite{Sznitman} change in this setting. We adopt the definition of the Wasserstein distance given in \cite[Eq. (1.4), p.173]{Sznitman}. 
\begin{proof}[Proof of Theorem \ref{th:smoothInter}]
To prove the claim, one should search for a fixed point of the map $\Phi:\Pp_T \to \Pp_T$ that to a given $m\in \Pp_T$ associates the law of the solution to the following SDE:
\begin{equation*}
\begin{cases}
&dX_t= dW_t  + b(t,X_t;(m_s)_{s\leq t}) dt,\\
&  X_0 \sim p_0.
\end{cases}
\end{equation*}
Notice that this equation is well-defined in strong sense thanks to \eqref{smoothLsigma} (see e.g. \cite[Thm. 5.2.9]{KaratzasShreve}). To exhibit the fixed point, the following contraction inequality should be shown for $m_1, m_2 \in \Pp_T$:\begin{equation}
\label{usefulLemma}
D_{1,t}(\Phi(m_1),\Phi(m_2)) \leq C_T \int_0^t  D_{1,u}(m_1,m_2) du.
\end{equation}
To prove the latter, follow the steps in \cite{Sznitman}. Always use (H0) when dealing with the time interaction.
 Denote by $X_1$ and $X_2$ the processes whose laws are $\Phi(m_1)$ and $\Phi(m_2)$.

Then, taking $\pi$ to be any coupling of $m_1$ and $m_2$, it comes
\begin{multline*}
    \E \ [ \sup_{s\leq t} |X_s^1 - X_s^2|] \leq \int_0^t \int_0^u \int_{\Cc\times \Cc} \Big |L(u-\alpha, X_u^1-w_\alpha^1)\\-L(u-\alpha, X_u^2-w_\alpha^2)\Big| d\pi (w^1,w^2) \ d\alpha \ du.
\end{multline*}
In view of (H0), one has
\begin{multline*}
   | L(u-\alpha, X_u^1-w_\alpha^1)-L(u-\alpha, X_u^2-w_\alpha^2)|\\ \leq (2 h_1(u-\alpha)+h_2(u-\alpha)) |X_u^1-w_\alpha^1- X_u^2+w_\alpha^2|\wedge 1.
\end{multline*}

Use that  $|X_u^1-X_u^2|\leq \sup_{r\leq u} |X_r^1-X_r^2|$ and $|w_\alpha^1-w_\alpha^2| \leq \sup_{r\leq u} |w_r^1-w_r^2|$ and apply Fubini's theorem in combination with integrability properties of $h_1$ and $h_2$. It comes
\begin{multline*}
  \E \ [ \sup_{s\leq t} |X_s^1 - X_s^2|]  \leq C_T \big [  \int_0^t \E [\sup_{r\leq u}  |X_r^1-X_r^2| \wedge 1]du \\+ \int_0^t  \int_{\Cc \times \Cc}\ \sup_{r\leq u} |w_r^1-w_r^2|\wedge 1 \  d\pi (w^1,w^2) du\big] .
\end{multline*}
Now, take an infimum over all couplings $\pi$ of $m_1$ and $m_2$. Afterwards, apply Gronwall's lemma. 
%
As $X^1$ and $X^2$ have laws $\Phi(m_1)$ and $\Phi(m_2)$, respectively, a standard  property of the Waserstein distance together with the preceding relation lead to the contraction inequality \eqref{usefulLemma}.
Once \eqref{usefulLemma} is obtained one repeats the arguments in~\cite{Sznitman} to finish the proof.
\end{proof}

\bibliography{biblio}

\begin{thebibliography}{10}

\bibitem{BilerCorrias}
{\sc Biler, P., Corrias, L., and Dolbeault, J.}
\newblock Large mass self-similar solutions of the parabolic-parabolic
  {K}eller-{S}egel model of chemotaxis.
\newblock {\em J. Math. Biol. 63}, 1 (2011), 1--32.

\bibitem{Blanchet}
{\sc Blanchet, A., Dolbeault, J., and Perthame, B.~t.}
\newblock Two-dimensional {K}eller-{S}egel model: optimal critical mass and
  qualitative properties of the solutions.
\newblock {\em Electron. J. Differential Equations\/} (2006), No. 44, 32.

\bibitem{Brezis}
{\sc Brezis, H.}
\newblock {\em Functional analysis, {S}obolev spaces and partial differential
  equations}.
\newblock Universitext. Springer, New York, 2011.

\bibitem{BrezisCazenave}
{\sc Brezis, H.~m., and Cazenave, T.}
\newblock A nonlinear heat equation with singular initial data.
\newblock {\em J. Anal. Math. 68\/} (1996), 277--304.

\bibitem{CalCor2008}
{\sc Calvez, V., and Corrias, L.}
\newblock The parabolic-parabolic {K}eller-{S}egel model in {$\Bbb R^2$}.
\newblock {\em Commun. Math. Sci. 6}, 2 (2008), 417--447.

\bibitem{Corrias2014}
{\sc Corrias, L., Escobedo, M., and Matos, J.}
\newblock Existence, uniqueness and asymptotic behavior of the solutions to the
  fully parabolic {K}eller-{S}egel system in the plane.
\newblock {\em J. Differential Equations 257}, 6 (2014), 1840--1878.

\bibitem{HerreroandVelazquez}
{\sc Herrero, M.~A., and Vel\'azquez, J. J.~L.}
\newblock A blow-up mechanism for a chemotaxis model.
\newblock {\em Ann. Scuola Norm. Sup. Pisa Cl. Sci. (4) 24}, 4 (1997), 633--683
  (1998).

\bibitem{HillenPotapov}
{\sc Hillen, T., and Potapov, A.}
\newblock The one-dimensional chemotaxis model: global existence and asymptotic
  profile.
\newblock {\em Math. Methods Appl. Sci. 27}, 15 (2004), 1783--1801.

\bibitem{Horstmann1}
{\sc Horstmann, D.}
\newblock From 1970 until present: the {K}eller-{S}egel model in chemotaxis and
  its consequences. {I}.
\newblock {\em Jahresber. Deutsch. Math.-Verein. 105}, 3 (2003), 103--165.

\bibitem{Horstmann2}
{\sc Horstmann, D.}
\newblock From 1970 until present: the {K}eller-{S}egel model in chemotaxis and
  its consequences. {II}.
\newblock {\em Jahresber. Deutsch. Math.-Verein. 106}, 2 (2004), 51--69.

\bibitem{JTT}
{\sc Jabir, J.-F., Talay, D., and Tomašević, M.}
\newblock Mean-field limit of a particle approximation of the one-dimensional
  parabolic-parabolic keller-segel model without smoothing.
\newblock {\em Electron. Commun. Probab. 23\/} (2018), 14 pp.

\bibitem{KaratzasShreve}
{\sc Karatzas, I., and Shreve, S.~E.}
\newblock {\em Brownian motion and stochastic calculus}, second~ed., vol.~113
  of {\em Graduate Texts in Mathematics}.
\newblock Springer-Verlag, New York, 1991.

\bibitem{KryRoc-05}
{\sc Krylov, N.~V., and R\"ockner, M.}
\newblock Strong solutions of stochastic equations with singular time dependent
  drift.
\newblock {\em Probab. Theory Related Fields 131}, 2 (2005), 154--196.

\bibitem{Mizo}
{\sc Mizoguchi, N.}
\newblock Global existence for the {C}auchy problem of the parabolic-parabolic
  {K}eller-{S}egel system on the plane.
\newblock {\em Calc. Var. Partial Differential Equations 48}, 3-4 (2013),
  491--505.

\bibitem{NagaiOgawa}
{\sc Nagai, T., and Ogawa, T.}
\newblock Global existence of solutions to a parabolic-elliptic system of
  drift-diffusion type in {$\bold{R}^2$}.
\newblock {\em Funkcial. Ekvac. 59}, 1 (2016), 67--112.

\bibitem{OsakiYagi}
{\sc Osaki, K., and Yagi, A.}
\newblock Finite dimensional attractor for one-dimensional {K}eller-{S}egel
  equations.
\newblock {\em Funkcial. Ekvac. 44}, 3 (2001), 441--469.

\bibitem{QianRussoZheng}
{\sc Qian, Z., Russo, F., and Zheng, W.}
\newblock Comparison theorem and estimates for transition probability densities
  of diffusion processes.
\newblock {\em Probab. Theory Related Fields 127}, 3 (2003), 388--406.

\bibitem{QianZheng}
{\sc Qian, Z., and Zheng, W.}
\newblock Sharp bounds for transition probability densities of a class of
  diffusions.
\newblock {\em C. R. Math. Acad. Sci. Paris 335}, 11 (2002), 953--957.

\bibitem{Sznitman}
{\sc Sznitman, A.-S.}
\newblock Topics in propagation of chaos.
\newblock In {\em \'{E}cole d'\'{E}t\'e de {P}robabilit\'es de {S}aint-{F}lour
  {XIX}---1989}, vol.~1464 of {\em Lecture Notes in Math.} Springer, Berlin,
  1991, pp.~165--251.

\bibitem{Mi-De}
{\sc Talay, D., and Toma\v{s}evi\'c, M.}
\newblock A new stochastic interpretation of {K}eller-{S}egel equations: the
  $1$-{D} case.
\newblock {\em Submitted, arXiv:1712.10254v3\/}.

\bibitem{Mi-De-2D}
{\sc Toma\v{s}evi\'c, M.}
\newblock On a prbabilistic interpretation of the parabolic-parabolic
  {K}eller-{S}egel model.
\newblock {\em Ph.{D}. {T}hesis\/} (2018).

\bibitem{Trevisan}
{\sc Trevisan, D.}
\newblock Well-posedness of multidimensional diffusion processes with weakly
  differentiable coefficients.
\newblock {\em Electron. J. Probab. 21\/} (2016), Paper No. 22, 41.

\end{thebibliography}
\end{document}